\documentclass[reqno]{amsart}
\usepackage{amssymb,latexsym,amsmath,amsthm,enumerate,amsbsy}
\usepackage[mathscr]{eucal}
\usepackage{framed,color,graphicx}
\usepackage{mathrsfs}
\usepackage[all]{xy}
\usepackage{tikz}
\usepackage{cite}
\usepackage{longtable}
\usepackage{ltxtable}
\usepackage{subcaption}
\usepackage{graphicx}
\usepackage{subfloat}
\usepackage{subfig}

\usetikzlibrary{positioning,decorations.pathreplacing,patterns,decorations.pathmorphing}
\tikzset{%
element/.style={draw, shape=circle, fill=white, inner sep=1.4pt}
}
\DeclareSymbolFont{bbold}{U}{bbold}{m}{n}
\DeclareSymbolFontAlphabet{\mathbbold}{bbold}

\theoremstyle{plain}
\newtheorem{thm}{Theorem}[section]
\newtheorem{lem}[thm]{Lemma}
\newtheorem{cor}[thm]{Corollary}
\newtheorem{pro}[thm]{Proposition}

\theoremstyle{definition}

\newtheorem{remark}[thm]{Remark}

\newcommand{\up}[1]{\textup{#1}}

\newcommand{\bp}{\mathbf{p}}
\newcommand{\bq}{\mathbf{q}}

\newcommand{\bu}{\mathbf{u}}
\newcommand{\bv}{\mathbf{v}}
\newcommand{\bw}{\mathbf{w}}

\begin{document}

\title[The finite basis problem for ai-semirings of order four]
{The finite basis problem for additively idempotent semirings of order four, I}

\author{Miaomiao Ren}
\address{School of Mathematics, Northwest University, Xi'an, 710127, Shaanxi, P.R. China}
\email{miaomiaoren@yeah.net}

\author{Junyang Liu}
\address{School of Mathematics, Northwest University, Xi'an, 710127, Shaanxi, P.R. China}
\email{ljy201817706@163.com}

\author{Lingli Zeng}
\address{School of Mathematics, Northwest University, Xi'an, 710127, Shaanxi, P.R. China}
\email{llzeng@yeah.net}

\author{Menglong Chen}
\address{School of Mathematics, Northwest University, Xi'an, 710127, Shaanxi, P.R. China}
\email{cml469937292@163.com}

\thanks{Miaomiao Ren is supported by National Natural Science Foundation of China (12371024).
Lingli Zeng, corresponding author, is supported by National Natural Science Foundation of China (11971383).
}

\begin{abstract}
We study the finite basis problem for $4$-element additively idempotent semirings whose
additive reducts are semilattices of height $1$. Up to isomorphism, there are $58$ such algebras.
We show that 49 of them are finitely based and the remaining ones are nonfinitely based.
\end{abstract}

\maketitle

\section{Introduction and preliminaries}
An \emph{additively idempotent semiring} (ai-semiring for short) is an algebra $(S, +, \cdot)$ with two binary operations $+$ and $\cdot$
such that the additive reduct $(S, +)$ is a semilattice, that is, a commutative idempotent semigroup,
the multiplicative reduct $(S, \cdot)$ is a semigroup and $S$ satisfies the distributive laws
\[
x(y+z)\approx xy+xz,\quad (x+y)z\approx xy+yz.
\]
Such an algebra is also called a \emph{semilattice-ordered semigroup}.
The class of ai-semirings contains
the Kleene semiring of regular languages \cite{con}, the max-plus algebra \cite{aei}
and the semiring of all binary relations on a set \cite{andmik}.
These algebras have played important roles in several branches of mathematics
such as algebraic geometry \cite{cc}, tropical geometry \cite{ms}, information science \cite{gl} and
theoretical computer science \cite{go}.

An ai-semiring is \emph{finitely based} if the set of its identities can be derived by some finite subset.
Otherwise, it is \emph{nonfinitely based}.
In the past two decades, the finite basis problem for ai-semirings has been intensively studied,
for example,
see \cite{dol07, dgv, gr, gmrz, gpz05, jrz, pas05, rjzl, rz16, rzs20, rzv, rzw, sr, shap23, vol21, wzr23, wrz23, yrzs, zrc}.
In particular, Pastijn et al. \cite{gpz05, pas05} showed that every
ai-semiring satisfying the identity $x^2\approx x$ is finitely based.
Ren et al. \cite{rz16, rzw} proved that every ai-semiring satisfying the identity
$x^3\approx x$ is finitely based.
Ren et al. \cite{rzs20} showed that every ai-semiring satisfying the identities
$x^n\approx x$ and $xy\approx yx$ is finitely based if $n-1$ is square-free.
Recently, Volkov et al. \cite{rzv} proved that if $n\geq 2$, then every ai-semiring
satisfying $x^n\approx x$ is finitely based if and only if $n=2$ or $3$.

\begin{table}[ht]
\caption{The Cayley tables of $S_7$} \label{tb2}
\begin{tabular}{c|ccc}
$+$&1&$a$&$\infty$\\
\hline
$1$&1&$\infty$&$\infty$\\
$a$&$\infty$&$a$&$\infty$\\
$\infty$&$\infty$&$\infty$&$\infty$\\
\end{tabular}\qquad
\begin{tabular}{c|ccc}
$\cdot$&1&$a$&$\infty$\\
\hline
$1$&1&$a$&$\infty$\\
$a$&$a$&$\infty$&$\infty$\\
$\infty$&$\infty$&$\infty$&$\infty$\\
\end{tabular}
\end{table}

On the other hand, several authors considered the finite basis problem for
ai-semirings of small order.
Dolinka \cite{dol07} provided the first example of a nonfinitely based finite ai-semiring,
which contains seven elements.
Shao and Ren \cite{sr} proved that every variety generated by some ai-semirings of order two is finitely based.
Zhao et al. \cite{zrc} showed that
with the possible exception of the ai-semiring $S_7$ (its Cayley tables are given by Table \ref{tb2}),
all ai-semirings of order three are finitely based.
Volkov \cite{vol21} answered the finite basis problem for the ai-semiring $B_2^1$ whose multiplicative reduct is the $6$-element Brandt monoid.
Jackson et al. \cite{jrz} presented some general results on the finite basis problem
for finite ai-semirings. As applications, they showed that
$S_7$ and $B_2^1$ are both nonfinitely based.
This completes the classification of ai-semirings of order three with respect to
the property of having/having no finite equational basis.
Recently, Gao et al. \cite{gmrz}, Shaprynski\v{\i} \cite{shap23} and Wu et al. \cite{wrz23}
initiated the study of the finite basis problem for ai-semirings of order four.

\setlength{\unitlength}{0.6cm}
\begin{figure}[ht]
\centering
\begin{subfigure}[b]{0.32\textwidth}
\centering
\begin{picture}(7,3)
\put(3.5,2.2){\line(0,-1){1}}
\put(3.5,2.2){\line(2,-1){2}}
\put(3.5,2.2){\line(-2,-1){2}}

\multiput(3.5,2.2)(0,-1){2}{\circle*{0.1}}
\multiput(3.5,2.2)(2,-1){2}{\circle*{0.1}}
\multiput(3.5,2.2)(-2,-1){2}{\circle*{0.1}}
\end{picture}
\caption{Type I} \label{Type I}
\end{subfigure}
\hfill
\begin{subfigure}[b]{0.32\textwidth}
\centering
\begin{picture}(7,3)
\put(3.5,2.2){\line(1,-1){1}}
\put(3.5,2.2){\line(-1,-1){1}}

\put(3.5,0.2){\line(1,1){1}}
\put(3.5,0.2){\line(-1,1){1}}

\multiput(3.5,2.2)(1,-1){2}{\circle*{0.1}}
\multiput(3.5,2.2)(-1,-1){2}{\circle*{0.1}}
\multiput(3.5,0.2)(1,1){2}{\circle*{0.1}}
\end{picture}
\caption{Type II} \label{Type II}
\end{subfigure}
\hfill
\begin{subfigure}[b]{0.32\textwidth}
\centering
\begin{picture}(7,3)
\put(3.5,2.2){\line(0,-1){1}}
\put(3.5,1.2){\line(1,-1){1}}
\put(3.5,1.2){\line(-1,-1){1}}

\multiput(3.5,2.2)(0,-1){2}{\circle*{0.1}}
\multiput(3.5,1.2)(1,-1){2}{\circle*{0.1}}
\multiput(3.5,1.2)(-1,-1){2}{\circle*{0.1}}
\end{picture}
\caption{Type III} \label{Type III}
\end{subfigure}

\begin{subfigure}[b]{0.495\textwidth}
\centering
\begin{picture}(6,3)
\put(3.3,2.2){\line(1,-1){1}}
\put(3.3,2.2){\line(-1,-1){1}}
\put(2.3,1.2){\line(0,-1){1}}

\multiput(3.3,2.2)(1,-1){2}{\circle*{0.1}}
\multiput(2.3,1.2)(0,-1){2}{\circle*{0.1}}
\end{picture}
\caption{Type IV} \label{Type IV}
\end{subfigure}
\hfill
\begin{subfigure}[b]{0.495\textwidth}
\centering
\begin{picture}(6,3)
\put(3.3,2.2){\line(0,-1){2}}
\multiput(3.3,2.2)(0,-1){1}{\circle*{0.1}}
\multiput(3.3,0.2)(0,-1){1}{\circle*{0.1}}
\multiput(3.3,1.54)(0,-1){1}{\circle*{0.1}}
\multiput(3.3,0.86)(0,-1){1}{\circle*{0.1}}
\end{picture}
\caption{Type V} \label{Type V}
\end{subfigure}
\caption{The additive orders of $4$-element ai-semirings}
\end{figure}

We follow this line of investigation and our aim is to
solve the finite basis problem for all ai-semirings of order four.
This program will be accomplished in a series of
papers, which we outline next.
The series will comprise five papers,
each addressing ai-semirings of order four characterized by a specific type
of additive semilattice.
Recall that the binary relation $\leq$ defined by
\[
a \leq b\Leftrightarrow a+b=b,
\]
is a partial order on every ai-semiring.
The present paper focuses on the finite basis problem
for $4$-element ai-semirings whose additive reducts are semilattices of height $1$ with respect to $\leq$ (see Figure \ref{Type I}).
The second paper \cite{yrzs} is devoted to studying the problem for
$4$-element ai-semirings whose additive reducts have the least element and two coatoms (see Figure \ref{Type II}).
The third paper will focus on the  problem for $4$-element ai-semirings
whose additive reducts have two minimal elements and one coatom (see Figure \ref{Type III}).
The fourth paper will study the  problem for $4$-element ai-semirings
whose additive reducts have two minimal elements and two coatoms (see Figure \ref{Type IV}).
The fifth paper will focus on the problem for $4$-element ai-semirings whose additive reducts are chains (see Figure \ref{Type V}).
Up to now, the third paper has been completed,
while there are still $4$ ai-semirings in the fourth paper and $31$ ai-semirings in the fifth paper
whose finite basis problem has not yet been solved.

A \emph{flat semiring} is an ai-semiring such that its multiplicative reduct has a zero element $0$
and $a+b=0$ for all distinct elements $a$ and $b$ of $S$.
So the additive reduct of a flat semiring is a semilattice of height $1$.
But an ai-semiring whose additive reduct is a semilattice of height $1$
is not necessarily a flat semiring. One can easily find many counterexamples in Table \ref{tb1} below.
It is easy to verify that the class of all flat semirings is closed under taking subalgebras and quotient algebras,
but is not closed under taking direct products. For example, the direct product of
two copies of the flat semiring $S_7$ is not a flat semiring.

Flat semirings have played an important role in the theory of varieties of ai-semirings (see \cite{jrz, rjzl, rzs20}).
Jackson et al. \cite[Lemma 2.2]{jrz} observed that a semigroup with the zero element $0$
becomes a flat semiring with the top element $0$ if and only if it is
$0$-cancellative, that is, $ab=ac\neq0$ implies $b=c$ and $ba=ca\neq0$ implies $b=c$
for all $a, b, c\in S$.
If $S$ is a cancellative semigroup, then $S^0$ is $0$-cancellative
and becomes a flat semiring, which is called the {\it flat extension} of $S$.
Let $G$ be a finite group. Jackson \cite[Theorem 7.3]{jac08} showed that
the flat extension of $G$ is finitely based if and only if all Sylow subgroups of $G$ are abelian.
Using the flat extensions of groups,
Ren et al. \cite{rjzl} provided an infinite series of minimal nonfinitely based ai-semiring varieties.

The following algebras form another important class of flat semirings.
Let $W$ be a nonempty subset of the free commutative semigroup $X^+_c$ over $X$,
and let $S_c(W)$ denote the set of all nonempty subwords of words in
$W$ together with a new symbol $0$. Define a binary operation $\cdot$ on $S_c(W)$ by the rule
\begin{equation*}
\bu\cdot \bv=
\begin{cases}
\bu\bv& \text{if }~\bu\bv\in S_c(W)\setminus \{0\}, \\
0& \text{otherwise.}
\end{cases}
\end{equation*}
Then $(S_c(W), \cdot)$ forms a semigroup with a zero element $0$.
It is easy to see that $(S_c(W), \cdot)$ is $0$-cancellative and so $S_c(W)$ becomes a flat semiring.
In particular, if $W$ consists of a single word $\bw$ we shall write $S_c(W)$ as $S_c(\bw)$.  If we allow the empty word in this construction, then the semigroup reduct is a monoid, and we use the notation $M_c(W)$.  Finally,
if we do the same construction on the free semigroup (resp., the free monoid) over $X$,
we obtain the flat semiring $S(W)$ (resp., $M(W)$).
Correspondingly, we have the notation $S(\bw)$ and $M(\bw)$.
It is easy to see that both $M_c(a)$ and $M(a)$ are isomorphic to $S_7$ if $a$ is a letter.

Let $X$ denote a countably infinite set of variables and $X^+$
the free semigroup over $X$. By distributivity, all ai-semiring terms over $X$ are finite sums of words in $X^+$.
An \emph{ai-semiring identity} over $X$ is an
expression of the form
\[
\bu\approx \bv,
\]
where $\bu$ and $\bv$ are ai-semiring terms over $X$.
From \cite[Theorem 2.5]{kp} we know that the ai-semiring
$(P_f(X^+), \cup, \cdot)$ consisting of all non-empty finite subsets of $X^+$
is free in the variety $\mathbf{AI}$ of all ai-semirings on $X$.
So we sometimes write
\[
\{\bu_i \mid 1 \leq i \leq k\}\approx \{\bv_j \mid 1 \leq j \leq \ell\}
\]
for the ai-semiring identity
\[
\bu_1+\cdots+\bu_k\approx \bv_1+\cdots+\bv_\ell.
\]
An \emph{ai-semiring substitution} is an endomorphism of $P_f(X^+)$.
Let $S$ be an ai-semiring and $\bu\approx \bv$ an ai-semiring identity.
We say that \emph{$S$ satisfies $\bu\approx \bv$} or \emph{$\bu\approx \bv$ holds in $S$}
if $\varphi(\bu)=\varphi(\bv)$ for all semiring homomorphisms $\varphi: P_f(X^+)\to S$.
Note that $\varphi$ is determined by $\{\varphi(x)\mid x\in X\}$,
since $P_f(X^+)$ is generated by $X$.

Suppose that $\Sigma$ is a set of ai-semiring identities which includes the identities that determine the variety of all ai-semirings.
Let $\bu \approx \bv$ be an ai-semiring identity such that
$\bu=\bu_1+\cdots+\bu_k$, $\bv=\bv_1+\cdots+\bv_\ell$, where $\bu_i$, $\bv_j\in X^+$, $1\leq i\leq k$, $1\leq j\leq \ell$.
Then it is easy to
see that the ai-semiring variety defined by $\bu\approx \bv$ is equal to the ai-semiring
variety defined by the simpler identities
$\bu\approx \bu+\bv_j, \bv\approx \bv+\bu_i, 1\leq i\leq k, 1\leq j\leq \ell$.
Therefore, to show that $\bu\approx \bv$ is derivable from $\Sigma$, we only need to show that
$\bu\approx \bu+\bv_j, \bv\approx \bv+\bu_i$ can be derived from $\Sigma$ for each $1\leq i\leq k$, $1\leq j\leq \ell$.
This technique will be repeatedly used in the sequel.

The following result about the equational logic of ai-semirings can be found in \cite[Lemma 2]{dol07}.
\begin{lem}\label{lem1}
Let $\Sigma$ be a set of ai-semiring identities and let $\bu\approx \bv$ be an ai-semiring identity.
Then $\bu\approx \bv$ is derivable from $\Sigma$ if and only if
there exist $T_1, T_2,\ldots, T_n \in P_f(X^+)$
such that $\bu=T_1$, $\bv=T_n$ and, for every $i=1, 2,\ldots, n-1$,
there are $A_i, B_i, P_i, Q_i, R_i\in P_f(X^+)$
and an ai-semiring substitution $\varphi_i: P_f(X^+) \to P_f(X^+)$ such that
\[
T_i=P_i\varphi_i(A_i)Q_i+R_i,~T_{i+1}=P_i\varphi_i(B_i)Q_i+R_i,
\]
where $A_i\approx B_i\in \Sigma$ or $B_i\approx A_i\in \Sigma$, $P_i$ and $Q_i$ may be the set $\{1\}$, $R_i$ may be the empty set.
\end{lem}

Next, we introduce some notation that will be repeatedly used in the sequel.
Let $\bw$ be a word in $X^+$ and $x$ a letter in $X$. Then
\begin{itemize}
\item $h(\bw)$ denotes the first variable that occurs in $\bw$;

\item $t(\bw)$ denotes the last variable that occurs in $\bw$;

\item $c(\bw)$ denotes the set of variables that occur in $\bw$;

\item $\ell(\bw)$ denotes the length of $\bw$, that is,
the number of variables occurring in $\bw$ counting multiplicities;

\item $m(x, \bw)$ denotes the number of occurrences of $x$ in $\bw$;

\item $p(\bw)$ denotes the word obtained from $\bw$ by deleting its tail,
that is, $\bw=p(\bw)t(\bw)$;

\item $s(\bw)$ denotes the word obtained from $\bw$ by deleting its head,
that is, $\bw=h(\bw)s(\bw)$.
\end{itemize}

Let $\bu$ be an ai-semiring term such that $\bu=\bu_1+\bu_2+\cdots+\bu_n$, where
$\bu_i \in X^+$, $1\leq i \leq n$. Then
\begin{itemize}
\item $h(\bu)$ denotes the set $\{h(\bu_i) \mid 1\leq i \leq n\}$;

\item $t(\bu)$ denotes the set $\{t(\bu_i) \mid 1\leq i \leq n\}$;

\item $c(\bu)$ denotes the set of variables that occur in $\bu$ and so
\[
c(\bu)=\bigcup_{1\leq i \leq n} c(\bu_i).
\]
\end{itemize}

Up to isomorphism, there are exactly 6 ai-semirings of order 2 (see \cite{sr}), which are denoted by $L_2$, $R_2$, $M_2$, $D_2$, $N_2$ and $T_2$. We assume that the carrier set of each
of these semirings is $\{0,1\}$. Their Cayley tables for addition and multiplication are listed in Table~\ref{2}.
It is easy to see that $M_2$ is isomorphic to $M(1)$ and that $T_2$ is isomorphic to $S(a)$,
where $1$ denotes the empty word and $a$ is a letter.
The following result, which can be found in \cite[Lemma 1.1]{sr},
provides a solution of the equational problem for all ai-semirings of order two.
It will be used without explicit reference.

\begin{table}[htbp]
\caption{The 2-element ai-semirings}
\label{2}
\begin{tabular}{cccccc}
\hline
Semiring&   $+$ &  $\cdot$ & Semiring &   $+$ &  $\cdot$\\
\hline
$L_2$&
\begin{tabular}{cc}
                    0 & 1  \\
                    1 & 1  \\
\end{tabular}&
\begin{tabular}{cc}
                    0 & 0  \\
                    1 & 1  \\
\end{tabular}
&
$R_2$&   \begin{tabular}{cc}
                    0 & 1  \\
                    1 & 1  \\
\end{tabular}&
\begin{tabular}{cc}
                    0 & 1  \\
                    0 & 1  \\
\end{tabular}
\\
\hline

$M_2$&
\begin{tabular}{cc}
                    0 & 1  \\
                    1 & 1  \\
\end{tabular}&
\begin{tabular}{cc}
                    0 & 1  \\
                    1 & 1  \\
\end{tabular}
&
$D_2$&   \begin{tabular}{cc}
                    0 & 1  \\
                    1 & 1  \\
\end{tabular}&
\begin{tabular}{cc}
                    0 & 0  \\
                    0 & 1  \\
\end{tabular}
\\
\hline
$N_2$&
\begin{tabular}{cc}
                    0 & 1  \\
                    1 & 1  \\
\end{tabular}&
\begin{tabular}{cc}
                    0 & 0  \\
                    0 & 0  \\
\end{tabular}
&
$T_2$&   \begin{tabular}{cc}
                    0 & 1  \\
                    1 & 1  \\
\end{tabular}&
\begin{tabular}{cc}
                    1 & 1  \\
                    1 & 1  \\
\end{tabular}
\\
\hline
\end{tabular}
\end{table}

\begin{lem}\label{nlemma1}
Let $\bu\approx \bu+\bq$ be an ai-semiring identity such that
$\bu=\bu_1+\cdots+\bu_n$, where $\bu_i, \bq\in X^+$, $1\leq i \leq n$. Then
\begin{itemize}
\item[$(1)$] $\bu\approx \bu+\bq$ holds in $L_2$ if and only if $h(\bq)=h(\bu_i)$ for some $\bu_i \in \bu$.

\item[$(2)$] $\bu\approx \bu+\bq$ holds in $R_2$ if and only if $t(\bq)=t(\bu_i)$ for some $\bu_i \in \bu$.

\item[$(3)$] $\bu\approx \bu+\bq$ holds in $M_2$ if and only if $c(\bq) \subseteq \bigcup_{i=1}^n c(\bu_i)$.

\item[$(4)$] $\bu\approx \bu+\bq$ holds in $D_2$ if and only if $c(\bq)\supseteq c(\bu_i)$ for some $\bu_i \in \bu$.

\item[$(5)$] $\bu\approx \bu+\bq$ holds in $N_2$ if and only if $\ell(\bq)\geq 2$ or $\ell(\bq)=1$,
$\bq=\bu_i$ for some $\bu_i \in \bu$.

\item[$(6)$] $\bu\approx \bu+\bq$ holds in $T_2$ if and only if $\ell(\bu_i)\geq 2$ for some $\bu_i \in \bu$
or $\ell(\bu_i)=1$ for all $\bu_i \in \bu$, $\bq=\bu_i$ for some $\bu_i \in \bu$.
\end{itemize}
\end{lem}

\begin{table}[htbp]
\caption{Some 3-element ai-semirings}\label{3-element ai-semirings}
\label{tb:classification}
\begin{tabular}{cccc}
\hline
Semiring&   $+$ &  $\cdot$ & Equational basis\\
\hline
$S_2$&   \begin{tabular}{ccc}
                    1 & 1 & 1  \\
                    1 & 2 & 1\\
                    1 & 1 & 3
              \end{tabular}&
\begin{tabular}{ccc}
                    1 & 1 & 1  \\
                    1 & 1 & 1\\
                    1 & 1 & 2
              \end{tabular}

&
\begin{tabular}{cc}
                    & $x_1x_2x_3\approx y_1y_2y_3$, $x+x^2\approx x^3$,  \\
                    & $x^2+y^2\approx xy$, $x^3+y \approx x^3$
\end{tabular}

\\
\hline
$S_4$&   \begin{tabular}{ccc}
                    1 & 1 & 1  \\
                    1 & 2 & 1\\
                    1 & 1 & 3
              \end{tabular}&
\begin{tabular}{ccc}
                    1 & 1 & 1  \\
                    1 & 1 & 1\\
                    1 & 2 & 3
              \end{tabular}
&
\begin{tabular}{cc}
                    & $xy\approx x^2y$, $xyz\approx yxz$,  \\
                    & $x+y^2\approx xy^2$, $x+yz \approx yx+yz$
\end{tabular}
\\
\hline

$S_5$&   \begin{tabular}{ccc}
                    1 & 1 & 1  \\
                    1 & 2 & 1\\
                    1 & 1 & 3
              \end{tabular}&
\begin{tabular}{ccc}
                    1 & 1 & 1  \\
                    1 & 1 & 1\\
                    3 & 3 & 3
              \end{tabular}
&
\begin{tabular}{cc}
                    & $xy\approx xz$, $xy\approx xy+x$,  \\
                    & $x+yz\approx x^2+yz$
\end{tabular}
\\
\hline
$S_6$&   \begin{tabular}{ccc}
                    1 & 1 & 1  \\
                    1 & 2 & 1\\
                    1 & 1 & 3
              \end{tabular}&
\begin{tabular}{ccc}
                    1 & 1 & 1  \\
                    1 & 1 & 2\\
                    1 & 1 & 3
              \end{tabular}
&
\begin{tabular}{cc}
                    & $xy\approx xy^2$, $xyz\approx xzy$,  \\
                    & $x+y^2\approx x^2y$, $x+yz \approx xz+yz$
\end{tabular}
\\
\hline
$S_{10}$&   \begin{tabular}{ccc}
                    1 & 1 & 1  \\
                    1 & 2 & 1\\
                    1 & 1 & 3
              \end{tabular}&
\begin{tabular}{ccc}
                    1 & 1 & 1  \\
                    1 & 2 & 3\\
                    1 & 3 & 2
              \end{tabular}
&$x^3\approx x$, $xy\approx yx$, $x^2+y^2\approx x^2y^2$\\
\hline
\end{tabular}
\end{table}

\begin{table}[htbp]
\caption{The multiplicative tables of $4$-element ai-semirings whose additive reduct are semilattices of height $1$} \label{tb1}
\begin{tabular}{cccccc}
\hline
Semiring & $\cdot$ & Semiring & $\cdot$ & Semiring & $\cdot$\\
\hline
$S_{(4, 1)}$
&
\begin{tabular}{cccc}
1 & 1 & 1 & 1\\
1 & 1 & 1 & 1 \\
1 & 1 & 1 & 1 \\
1 & 1 & 1 & 1 \\
\end{tabular}
&
$S_{(4, 2)}$
&
\begin{tabular}{cccc}
1 & 1 & 1 & 1\\
1 & 1 & 1 & 1 \\
1 & 1 & 1 & 1 \\
1 & 1 & 1 & 2 \\
\end{tabular}
&
$S_{(4, 3)}$
&
\begin{tabular}{cccc}
1 & 1 & 1 & 1\\
1 & 1 & 1 & 1 \\
1 & 1 & 1 & 1 \\
1 & 1 & 1 & 4 \\
\end{tabular}\\
\hline

$S_{(4, 4)}$
&
\begin{tabular}{cccc}
1 & 1 & 1 & 1\\
1 & 1 & 1 & 1 \\
1 & 1 & 1 & 2 \\
1 & 1 & 1 & 1 \\
\end{tabular}
&
$S_{(4, 5)}$
&
\begin{tabular}{cccc}
1 & 1 & 1 & 1\\
1 & 1 & 1 & 1 \\
1 & 1 & 1 & 3 \\
1 & 1 & 1 & 4 \\
\end{tabular}
&
$S_{(4, 6)}$
&
\begin{tabular}{cccc}
1 & 1 & 1 & 1\\
1 & 1 & 1 & 2 \\
1 & 1 & 1 & 3 \\
1 & 1 & 1 & 4 \\
\end{tabular}\\
\hline

$S_{(4, 7)}$
&
\begin{tabular}{cccc}
1 & 1 & 1 & 4\\
1 & 1 & 1 & 4 \\
1 & 1 & 1 & 4 \\
1 & 1 & 1 & 4 \\
\end{tabular}
&
$S_{(4, 8)}$
&
\begin{tabular}{cccc}
1 & 1 & 1 & 1\\
1 & 1 & 1 & 1 \\
1 & 1 & 1 & 2 \\
1 & 1 & 2 & 1 \\
\end{tabular}
&
$S_{(4, 9)}$
&
\begin{tabular}{cccc}
1 & 1 & 1 & 1\\
1 & 1 & 1 & 1 \\
1 & 1 & 1 & 2 \\
1 & 1 & 2 & 3 \\
\end{tabular}\\
\hline

$S_{(4, 10)}$
&
\begin{tabular}{cccc}
1 & 1 & 1 & 1\\
1 & 1 & 1 & 1 \\
1 & 1 & 1 & 1 \\
1 & 1 & 3 & 4 \\
\end{tabular}
&
$S_{(4, 11)}$
&
\begin{tabular}{cccc}
1 & 1 & 1 & 1\\
1 & 1 & 1 & 1 \\
1 & 1 & 1 & 3 \\
1 & 1 & 3 & 4 \\
\end{tabular}
&
$S_{(4, 12)}$
&
\begin{tabular}{cccc}
1 & 1 & 1 & 1\\
1 & 1 & 1 & 2 \\
1 & 1 & 1 & 1 \\
1 & 1 & 3 & 4 \\
\end{tabular}\\
\hline
$S_{(4, 13)}$
&
\begin{tabular}{cccc}
1 & 1 & 1 & 1\\
1 & 1 & 1 & 2 \\
1 & 1 & 1 & 3 \\
1 & 1 & 3 & 4 \\
\end{tabular}
&
$S_{(4, 14)}$
&
\begin{tabular}{cccc}
1 & 1 & 1 & 1\\
1 & 1 & 1 & 1 \\
1 & 1 & 2 & 1 \\
1 & 1 & 1 & 2 \\
\end{tabular}
&
$S_{(4, 15)}$
&
\begin{tabular}{cccc}
1 & 1 & 1 & 1\\
1 & 1 & 1 & 1 \\
1 & 1 & 2 & 1 \\
1 & 1 & 1 & 4 \\
\end{tabular}\\
\hline

$S_{(4, 16)}$
&
\begin{tabular}{cccc}
1 & 1 & 1 & 4\\
1 & 1 & 1 & 4 \\
1 & 1 & 2 & 4 \\
1 & 1 & 1 & 4 \\
\end{tabular}
&
$S_{(4, 17)}$
&
\begin{tabular}{cccc}
1 & 1 & 1 & 1\\
1 & 1 & 1 & 1 \\
1 & 1 & 3 & 1 \\
1 & 1 & 1 & 4 \\
\end{tabular}
&
$S_{(4, 18)}$
&
\begin{tabular}{cccc}
1 & 1 & 1 & 1\\
1 & 1 & 1 & 2 \\
1 & 1 & 3 & 1 \\
1 & 1 & 1 & 4 \\
\end{tabular}\\
\hline

$S_{(4, 19)}$
&
\begin{tabular}{cccc}
1 & 1 & 1 & 4\\
1 & 1 & 1 & 4 \\
1 & 1 & 3 & 4 \\
1 & 1 & 1 & 4 \\
\end{tabular}
&
$S_{(4, 20)}$
&
\begin{tabular}{cccc}
1 & 1 & 1 & 1\\
1 & 1 & 1 & 1 \\
1 & 1 & 3 & 4 \\
1 & 1 & 4 & 3 \\
\end{tabular}
&
$S_{(4, 21)}$
&
\begin{tabular}{cccc}
1 & 1 & 1 & 4\\
1 & 1 & 2 & 4 \\
1 & 1 & 3 & 4 \\
1 & 1 & 1 & 4 \\
\end{tabular}\\
\hline

$S_{(4, 22)}$
&
\begin{tabular}{cccc}
1 & 1 & 3 & 4\\
1 & 1 & 3 & 4 \\
1 & 1 & 3 & 4 \\
1 & 1 & 3 & 4 \\
\end{tabular}
&
$S_{(4, 23)}$
&
\begin{tabular}{cccc}
1 & 1 & 1 & 1\\
1 & 1 & 1 & 1 \\
1 & 1 & 1 & 1 \\
1 & 2 & 3 & 4 \\
\end{tabular}
&
$S_{(4, 24)}$
&
\begin{tabular}{cccc}
1 & 1 & 1 & 1\\
1 & 1 & 1 & 1 \\
1 & 1 & 1 & 3 \\
1 & 2 & 3 & 4 \\
\end{tabular}\\
\hline

$S_{(4, 25)}$
&
\begin{tabular}{cccc}
1 & 1 & 1 & 1\\
1 & 1 & 1 & 2 \\
1 & 1 & 1 & 3 \\
1 & 2 & 3 & 4 \\
\end{tabular}
&
$S_{(4, 26)}$
&
\begin{tabular}{cccc}
1 & 1 & 1 & 1\\
1 & 1 & 1 & 2 \\
1 & 1 & 2 & 3 \\
1 & 2 & 3 & 4 \\
\end{tabular}
&
$S_{(4, 27)}$
&
\begin{tabular}{cccc}
1 & 1 & 1 & 1\\
1 & 1 & 1 & 1 \\
1 & 1 & 3 & 1 \\
1 & 2 & 1 & 4 \\
\end{tabular}\\
\hline

$S_{(4, 28)}$
&
\begin{tabular}{cccc}
1 & 1 & 1 & 1\\
1 & 1 & 1 & 2 \\
1 & 1 & 3 & 1 \\
1 & 2 & 1 & 4 \\
\end{tabular}
&
$S_{(4, 29)}$
&
\begin{tabular}{cccc}
1 & 1 & 1 & 1\\
1 & 1 & 2 & 1 \\
1 & 1 & 3 & 1 \\
1 & 2 & 1 & 4 \\
\end{tabular}
&
$S_{(4, 30)}$
&
\begin{tabular}{cccc}
1 & 1 & 3 & 1\\
1 & 1 & 3 & 1 \\
1 & 1 & 3 & 1 \\
1 & 2 & 3 & 4 \\
\end{tabular}\\
\hline

$S_{(4, 31)}$
&
\begin{tabular}{cccc}
1 & 1 & 3 & 1\\
1 & 1 & 3 & 2 \\
1 & 1 & 3 & 1 \\
1 & 2 & 3 & 4 \\
\end{tabular}
&
$S_{(4, 32)}$
&
\begin{tabular}{cccc}
1 & 1 & 1 & 1\\
1 & 2 & 1 & 1 \\
1 & 1 & 3 & 1 \\
1 & 1 & 1 & 4 \\
\end{tabular}
&
$S_{(4, 33)}$
&
\begin{tabular}{cccc}
1 & 1 & 1 & 4 \\
1 & 2 & 1 & 4 \\
1 & 1 & 3 & 4 \\
1 & 1 & 1 & 4 \\
\end{tabular}\\
\hline
\end{tabular}
\end{table}

\begin{table}[htbp]
\begin{tabular}{cccccc}
\hline
$S_{(4, 34)}$
&
\begin{tabular}{cccc}
1 & 1 & 1 & 1\\
1 & 2 & 1 & 1 \\
1 & 1 & 3 & 4 \\
1 & 1 & 4 & 3 \\
\end{tabular}
&
$S_{(4, 35)}$
&
\begin{tabular}{cccc}
1 & 1 & 3 & 4 \\
1 & 2 & 3 & 4 \\
1 & 1 & 3 & 4 \\
1 & 1 & 3 & 4 \\
\end{tabular}
&
$S_{(4, 36)}$
&
\begin{tabular}{cccc}
1 & 1 & 3 & 1 \\
1 & 2 & 3 & 4 \\
1 & 1 & 3 & 1 \\
1 & 4 & 3 & 2 \\
\end{tabular}\\
\hline
$S_{(4, 37)}$
&
\begin{tabular}{cccc}
1 & 1 & 1 & 1\\
1 & 2 & 3 & 4 \\
1 & 3 & 4 & 2 \\
1 & 4 & 2 & 3 \\
\end{tabular}
&
$S_{(4, 38)}$
&
\begin{tabular}{cccc}
1 & 2 & 3 & 4 \\
1 & 2 & 3 & 4 \\
1 & 2 & 3 & 4 \\
1 & 2 & 3 & 4 \\
\end{tabular}
&
$S_{(4, 39)}$
&
\begin{tabular}{cccc}
1 & 1 & 1 & 1 \\
1 & 1 & 1 & 1 \\
1 & 1 & 1 & 1 \\
4 & 4 & 4 & 4 \\
\end{tabular}\\
\hline

$S_{(4, 40)}$
&
\begin{tabular}{cccc}
1 & 1 & 1 & 4\\
1 & 1 & 1 & 4 \\
1 & 1 & 1 & 4 \\
4 & 4 & 4 & 4 \\
\end{tabular}
&
$S_{(4, 41)}$
&
\begin{tabular}{cccc}
1 & 1 & 1 & 1 \\
1 & 1 & 1 & 1 \\
1 & 1 & 2 & 1 \\
4 & 4 & 4 & 4 \\
\end{tabular}
&
$S_{(4, 42)}$
&
\begin{tabular}{cccc}
1 & 1 & 1 & 4 \\
1 & 1 & 1 & 4 \\
1 & 1 & 2 & 4 \\
4 & 4 & 4 & 4 \\
\end{tabular}\\
\hline

$S_{(4, 43)}$
&
\begin{tabular}{cccc}
1 & 1 & 1 & 1\\
1 & 1 & 1 & 1 \\
1 & 1 & 3 & 1 \\
4 & 4 & 4 & 4 \\
\end{tabular}
&
$S_{(4, 44)}$
&
\begin{tabular}{cccc}
1 & 1 & 1 & 4 \\
1 & 1 & 1 & 4 \\
1 & 1 & 3 & 4 \\
4 & 4 & 4 & 4 \\
\end{tabular}
&
$S_{(4, 45)}$
&
\begin{tabular}{cccc}
1 & 1 & 1 & 1 \\
1 & 1 & 2 & 1 \\
1 & 1 & 3 & 1 \\
4 & 4 & 4 & 4 \\
\end{tabular}\\
\hline

$S_{(4, 46)}$
&
\begin{tabular}{cccc}
1 & 1 & 1 & 4\\
1 & 1 & 2 & 4 \\
1 & 1 & 3 & 4 \\
4 & 4 & 4 & 4 \\
\end{tabular}
&
$S_{(4, 47)}$
&
\begin{tabular}{cccc}
1 & 1 & 1 & 1 \\
1 & 1 & 1 & 1 \\
1 & 2 & 3 & 1 \\
4 & 4 & 4 & 4 \\
\end{tabular}
&
$S_{(4, 48)}$
&
\begin{tabular}{cccc}
1 & 1 & 1 & 4 \\
1 & 1 & 1 & 4 \\
1 & 2 & 3 & 4 \\
4 & 4 & 4 & 4 \\
\end{tabular}\\
\hline

$S_{(4, 49)}$
&
\begin{tabular}{cccc}
1 & 1 & 1 & 1\\
1 & 1 & 2 & 1 \\
1 & 2 & 3 & 1 \\
4 & 4 & 4 & 4 \\
\end{tabular}
&
$S_{(4, 50)}$
&
\begin{tabular}{cccc}
1 & 1 & 1 & 4 \\
1 & 1 & 2 & 4 \\
1 & 2 & 3 & 4 \\
4 & 4 & 4 & 4 \\
\end{tabular}
&
$S_{(4, 51)}$
&
\begin{tabular}{cccc}
1 & 1 & 1 & 1 \\
1 & 2 & 1 & 1 \\
1 & 1 & 3 & 1 \\
4 & 4 & 4 & 4 \\
\end{tabular}\\
\hline

$S_{(4, 52)}$
&
\begin{tabular}{cccc}
1 & 1 & 1 & 4\\
1 & 2 & 1 & 4 \\
1 & 1 & 3 & 4 \\
4 & 4 & 4 & 4 \\
\end{tabular}
&
$S_{(4, 53)}$
&
\begin{tabular}{cccc}
1 & 1 & 1 & 1 \\
1 & 2 & 3 & 1 \\
1 & 3 & 2 & 1 \\
4 & 4 & 4 & 4 \\
\end{tabular}
&
$S_{(4, 54)}$
&
\begin{tabular}{cccc}
1 & 1 & 1 & 4 \\
1 & 2 & 3 & 4 \\
1 & 3 & 2 & 4 \\
4 & 4 & 4 & 4 \\
\end{tabular}\\
\hline

$S_{(4, 55)}$
&
\begin{tabular}{cccc}
1 & 1 & 1 & 1\\
1 & 1 & 1 & 1 \\
3 & 3 & 3 & 3 \\
4 & 4 & 4 & 4 \\
\end{tabular}
&
$S_{(4, 56)}$
&
\begin{tabular}{cccc}
1 & 1 & 1 & 1 \\
1 & 2 & 1 & 1 \\
3 & 3 & 3 & 3 \\
4 & 4 & 4 & 4 \\
\end{tabular}
&
$S_{(4, 57)}$
&
\begin{tabular}{cccc}
1 & 1 & 1 & 4 \\
2 & 2 & 2 & 2 \\
3 & 3 & 3 & 3 \\
4 & 4 & 4 & 4 \\
\end{tabular}\\
\hline

$S_{(4, 58)}$
&
\begin{tabular}{cccc}
2 & 2 & 2 & 2 \\
2 & 2 & 2 & 2 \\
2 & 2 & 2 & 2 \\
2 & 2 & 2 & 2 \\
\end{tabular}\\
\hline
\end{tabular}
\end{table}

Up to isomorphism, there are 61 ai-semirings of order three,
which are denoted by $S_i$, $1 \leq i \leq 61$.
For detailed information on these semirings, the readers can refer to \cite{zrc}.
Table \ref{3-element ai-semirings} lists Cayley tables and equational basis for
$S_2$, $S_4$, $S_5$, $S_6$ and $S_{10}$,
which will be heavily used in the present paper.
 We assume that the carrier set of each
of these semirings is $\{1, 2, 3\}$.
Up to isomorphism,
there are exactly 866 ai-semirings of order four\footnote{We wrote a program and obtained this result.
Dataset link: https://github.com/cml-daishu/n-element-ai-semirings-solver
The readers can also refer to the program of James Mitchell.
Dataset link: https://github.com/james-d-mitchell/ai-semirings
}.
The additive reducts of $58$ of them are semilattices of height $1$.
These algebras are denoted by $S_{(4, k)}$, $1\leq k\leq 58$.
We assume that the carrier set of each of these semirings is $\{1, 2, 3, 4\}$, where
$1$ is the greatest element of the additive reduct of each semiring.
Their Cayley tables for multiplication are listed in Table \ref{tb1},
and their Cayley tables for addition are determined by Figure \ref{figure1}.
\setlength{\unitlength}{1cm}
\setlength{\abovecaptionskip}{5pt}
\begin{figure}[ht]
\begin{picture}(35,1.6)
\put(6.5,1.6){\line(0,-1){1}}
\put(6.5,1.6){\line(2,-1){2}}
\put(6.5,1.6){\line(-2,-1){2}}

\multiput(6.5,1.6)(0,-1){2}{\circle*{0.1}}
\multiput(6.5,1.6)(2,-1){2}{\circle*{0.1}}
\multiput(6.5,1.6)(-2,-1){2}{\circle*{0.1}}

\put(6.5,1.9){\makebox(0,0){$1$}}
\put(6.5,0.3){\makebox(0,0){$3$}}
\put(4.5,0.3){\makebox(0,0){$2$}}
\put(8.5,0.3){\makebox(0,0){$4$}}
\end{picture}
\caption{The additive order of $S_{(4, k)}$, $1\leq k\leq 58$} \label{figure1}
\end{figure}

The following theorem is the main result of this paper.
\begin{thm}\label{main}
\hspace*{\fill}
\begin{itemize}
\item[$(1)$] The following ai-semirings are finitely based\up:
$S_{(4, 1)}$, $S_{(4, 2)}$, $S_{(4, 3)}$,
$S_{(4, 4)}$, $S_{(4, 5)}$, $S_{(4, 6)}$, $S_{(4, 7)}$, $S_{(4, 8)}$, $S_{(4, 9)}$, $S_{(4, 10)}$,
$S_{(4, 12)}$,
$S_{(4, 14)}$, $S_{(4, 15)}$, $S_{(4, 16)}$, $S_{(4, 17)}$, $S_{(4, 18)}$, $S_{(4, 19)}$, $S_{(4, 20)}$,
$S_{(4, 21)}$, $S_{(4, 22)}$, $S_{(4, 23)}$,
$S_{(4, 27)}$, $S_{(4, 29)}$, $S_{(4, 30)}$,
$S_{(4, 32)}$, $S_{(4, 33)}$, $S_{(4, 34)}$,
$S_{(4, 35)}$, $S_{(4, 36)}$, $S_{(4, 37)}$,
$S_{(4, 38)}$, $S_{(4, 39)}$, $S_{(4, 40)}$,
$S_{(4, 41)}$, $S_{(4, 42)}$, $S_{(4, 43)}$,
$S_{(4, 44)}$, $S_{(4, 45)}$, $S_{(4, 46)}$,
$S_{(4, 47)}$, $S_{(4, 48)}$, $S_{(4, 51)}$,
$S_{(4, 52)}$, $S_{(4, 53)}$, $S_{(4, 54)}$,
$S_{(4, 55)}$, $S_{(4, 56)}$, $S_{(4, 57)}$ and $S_{(4, 58)}$.

\item[$(2)$] The following ai-semirings are nonfinitely based\up:
$S_{(4, 11)}$, $S_{(4, 13)}$, $S_{(4, 24)}$, $S_{(4, 25)}$,
$S_{(4, 26)}$, $S_{(4, 28)}$, $S_{(4, 31)}$, $S_{(4, 49)}$ and $S_{(4, 50)}$.
\end{itemize}
\end{thm}

The proof of Theorem \ref{main} will be completed in the following five sections.

\section{Flat semirings}
In this section we provide the proof for Theorem \ref{main} (2).
Let $S$ be a finite ai-semiring. Then $\mathsf{V}(S)$ denotes the variety generated by $S$,
that is, the smallest variety that contains $S$.
An element $a$ of $S$
is $\emph{cyclic}$ if $a^n=a$ for some $n>1$.
The \emph{index} of $S$ is the smallest $k$
such that $S$ satisfies the identity $x^k\approx x^{k+\ell}$ for some $\ell\geq 1$.
The following result is due to Jackson et al. \cite[Theorem 4.9]{jrz}.
\begin{lem}\label{ln1}
Let $S$ be a finite ai-semiring whose noncyclic elements form an order ideal,
and let $k'$ denote the index of $S$.
If $S_c(a_1\cdots a_k)$ lies in $\mathsf{V}(S)$
for some $k\geq \max(k', 3)$, then $S$ is nonfinitely based.
\end{lem}

From \cite[Proposition 2.6]{jrz} we know that
the variety $\mathsf{V}(S_7)$ contains $S_c(a_1\cdots a_k)$ for all $k\geq 1$.
So by Lemma \ref{ln1} we immediately deduce
\begin{pro}\label{pro1}
Let $S$ be a finite ai-semiring whose noncyclic elements form an order ideal.
If $\mathsf{V}(S)$ contains the semiring $S_7$, then $S$ is nonfinitely based.
\end{pro}

\begin{cor}\label{cor24070601}
$S_{(4, 11)}$, $S_{(4, 13)}$, $S_{(4, 24)}$, $S_{(4, 25)}$,
$S_{(4, 26)}$, $S_{(4, 28)}$, $S_{(4, 31)}$, $S_{(4, 49)}$ and $S_{(4, 50)}$
are all nonfinitely based.
\end{cor}
\begin{proof}
Let $S$ be a semiring in Corollary \ref{cor24070601}.
It is easy to see that the set of noncyclic elements of $S$ forms an order ideal
and that $S$ contains a copy of $S_7$.
By Proposition \ref{pro1} we have that $S$ is nonfinitely based.
\end{proof}

\begin{cor}\label{l1}
Let $\mathcal{V}$ be a variety generated by a finite family of finite flat semirings.
If $\mathcal{V}$ contains the semiring $S_7$, then it is nonfinitely based.
\end{cor}
\begin{proof}
Suppose that $\mathcal{V}$ is a variety that contains $S_7$ and is generated by finite flat semirings
$S_i, 1\leq i \leq n$. Then $\mathcal{V}$ is generated by the finite direct product $S_1\times \cdots\times S_n$.
Assume that $(a_1, \ldots, a_n)$ is an arbitrary cyclic element of $S_1\times \cdots\times S_n$.
Then there exists $k> 1$ such that
\[
(a^k_1, \ldots, a^k_n)=(a_1, \ldots, a_n)^k=(a_1, \ldots, a_n)
\]
and so $a^k_i=a_i$ for all $1 \leq i \leq n$. If
$(a_1, \ldots, a_n)\leq (b_1, \ldots, b_n)$ for some element $(b_1, \ldots, b_n)$ of $S_1\times \cdots\times S_n$,
then $a_i\leq b_i$ for all $1\leq i\leq n$.
Since every $S_i$ is a flat semiring, it follows that $b_i=a_i$ or $0$ for $1\leq i\leq n$.
Thus $b_i^k=b_i$ for all $1\leq i\leq n$ and so
\[
(b_1, \ldots, b_n)^k=(b^k_1, \ldots, b^k_n)=(b_1, \ldots, b_n).
\]
This shows that $(b_1, \ldots, b_n)$ is a cyclic element.
So the set of all noncyclic elements of $S_1\times \cdots\times S_n$ forms an order ideal.
By Proposition \ref{pro1} we immediately deduce that $\mathcal{V}$ is nonfinitely based.
\end{proof}

\begin{remark}
Corollary \ref{l1} has been mentioned in \cite[P. 235]{jrz}. We present a proof here.
\end{remark}

Let $n$ be a positive integer and $\mathbf{F}_n$ denote the variety generated by all flat semirings
of order $n$.

\begin{thm}
Let $n\geq 1$. Then
$\mathbf{F}_n$ is finitely based if and only if $n\leq 2$.
\end{thm}
\begin{proof}
If $n\leq 2$, then by the main result of \cite{sr}, we can deduce that $\mathbf{F}_n$ is finitely based.
Suppose that $n\geq 3$. Let $a$ be a letter.
Then $M(a^{n-2})$ is a flat semiring of order $n$ and
so it is a member of $\mathbf{F}_n$.
Since $S_7$ is isomorphic to $\{1, a^{n-2}, 0\}$,
it follows that that $S_7$ lies in $\mathbf{F}_n$.
By Corollary \ref{l1} we deduce that $\mathbf{F}_n$ is nonfinitely based as required.
\end{proof}

In the remainder of this section we introduce two constructions on flat semirings.
Let $S$ be a flat semiring. If we adjoin an extra element $b$ to $S$
and define
\[
bb=ba=ab=0
\]
for all $a\in S$, then $(S \cup \{b\}, \cdot)$ is a $0$-cancellative
semigroup and so $S \cup \{b\}$ becomes a flat semiring. It will be called the \emph{null extension} of $S$ and is
denoted by $S_{ne}$. One can easily verify that $S_{ne}$ is isomorphic to a subdirect product of $S$ and $T_2$.
So we have

\begin{pro}\label{lnn4}
Let $S$ be a flat semiring. Then
$\mathsf{V}(S_{ne})$ is the join of $\mathsf{V}(S)$ and $\mathsf{V}(T_2)$.
In particular, if $\mathsf{V}(S)$ contains $T_2$, then $\mathsf{V}(S_{ne})=\mathsf{V}(S)$.
\end{pro}

\begin{pro}\label{pro24071001}
If $n\geq 1$, then $\mathbf{F}_n$ is a proper subvariety of $\mathbf{F}_{n+1}$.
\end{pro}
\begin{proof}
Let $S$ be a flat semiring of order $n$. Then $S_{ne}$ is
a flat semiring of order $n+1$ and contains $S$.
Since $S_{ne}$ is a member of $\mathbf{F}_{n+1}$,
it follows that $S$ lies in $\mathbf{F}_{n+1}$.
So $\mathbf{F}_n$ is a subvariety of $\mathbf{F}_{n+1}$.
On the other hand, it is easy to see that the identity
\[
x^n \approx x^{n+n!}
\]
is satisfied by every flat semiring of order $n$
and so does hold in $\mathbf{F}_n$, but is not satisfied by $\mathbf{F}_{n+1}$,
since the flat semiring $S(a^n)$ of order $n+1$ does not satisfy it.
Thus $\mathbf{F}_n$ is a proper subvariety of $\mathbf{F}_{n+1}$ as required.
\end{proof}

Let $S$ be a flat semiring. If we adjoin an extra element $e$ to $S$
and define
\[
e^2=e, ~ea=ae=0
\]
for all $a\in S\setminus \{e\}$, then $(S \cup \{e\}, \cdot)$ is a $0$-cancellative
semigroup and so $S \cup \{e\}$ becomes a flat semiring. It will be called the \emph{idempotent extension} of $S$ and is
denoted by $S_{ie}$. It is easy to see that $S_{ie}$ is isomorphic to a subdirect product of $S$ and $M_2$.
Thus we have

\begin{pro}\label{pro24070610}
Let $S$ be a flat semiring. Then
$\mathsf{V}(S_{ie})$ is the join of $\mathsf{V}(S)$ and $\mathsf{V}(M_2)$.
In particular, if $\mathsf{V}(S)$ contains $M_2$, then $\mathsf{V}(S_{ie})=\mathsf{V}(S)$.
\end{pro}

An ai-semiring identity $\bu \approx \bv$ is \emph{regular} if $c(\bu)=c(\bv)$.
Let $\mathcal{V}$ and $\mathcal{W}$ be ai-semiring varieties.
Then $\mathcal{W}$ is the \emph{regularization} of $\mathcal{V}$
if $\mathcal{W}$ is the join of $\mathcal{V}$ and $\mathsf{V}(M_2)$,
that is, $\mathcal{W}$ is defined by all regular identities of $\mathcal{V}$.
By Proposition \ref{pro24070610} we know that the variety $\mathsf{V}(S_{ie})$ is determined by
all regular identities that are satisfied by $S$. So
$\mathsf{V}(S_{ie})$ is the regularization of $\mathsf{V}(S)$.

\section{Equational basis of $S_{(4, 4)}$, $S_{(4, 14)}$ and $S_{(4, 20)}$}
In this section we first show that some $4$-element ai-semirings are finitely based by some known results,
and then give finite equational basis for $S_{(4, 4)}$, $S_{(4, 14)}$, and $S_{(4, 20)}$
from the syntactic approach.

\begin{pro}\label{pro24070701}
The following ai-semirings are finitely based: $S_{(4, 1)}$, $S_{(4, 3)}$, $S_{(4, 7)}$,
$S_{(4, 17)}$, $S_{(4, 19)}$, $S_{(4, 22)}$, $S_{(4, 32)}$, $S_{(4, 33)}$, $S_{(4, 35)}$,
$S_{(4, 38)}$, $S_{(4, 39)}$, $S_{(4, 40)}$, $S_{(4, 43)}$, $S_{(4, 44)}$, $S_{(4, 51)}$,
$S_{(4, 52)}$, $S_{(4, 55)}$, $S_{(4, 56)}$, $S_{(4, 57)}$ and $S_{(4, 58)}$.
\end{pro}
\begin{proof}
It is easy to verify that every semiring in Proposition \ref{pro24070701}
satisfies the equational basis (see \cite[Theorem 2.1]{sr}) of the variety generated by all ai-semirings of order two.
By the main result of \cite{sr} these algebras are all finitely based.
\end{proof}

\begin{pro}\label{pro24070702}
The following ai-semirings are finitely based: $S_{(4, 34)}$, $S_{(4, 36)}$, $S_{(4, 53)}$ and $S_{(4, 54)}$.
\end{pro}
\begin{proof}
It is easy to see that every semiring in Proposition \ref{pro24070702} satisfies
$x^3\approx x$. From the main result of \cite{rzw} we deduce that these semirings are all finitely based.
\end{proof}

\begin{pro}
The following ai-semirings are finitely based: $S_{(4, 2)}$, $S_{(4, 5)}$,
$S_{(4, 6)}$, $S_{(4, 8)}$, $S_{(4, 9)}$, $S_{(4, 10)}$,  $S_{(4, 18)}$,
$S_{(4, 23)}$, $S_{(4, 27)}$, $S_{(4, 29)}$, $S_{(4, 37)}$.
\end{pro}
\begin{proof}
It is easy to see that $S_{(4, 2)}$ contains a copy of $S_2$ and
satisfies the finite equational basis (see \cite[Proposition 2]{zrc}) of $S_2$.
This implies that $\mathsf{V}(S_{(4, 2)})=\mathsf{V}(S_2)$
and so $S_{(4, 2)}$ is finitely based.
One can use the same approach to prove that
\[
\mathsf{V}(S_{(4, 5)})=\mathsf{V}(S_{(4, 18)})=\mathsf{V}(S_6)
\]
and
\[
\mathsf{V}(S_{(4, 10)})=\mathsf{V}(S_{(4, 23)})=\mathsf{V}(S_{(4, 27)})=\mathsf{V}(S_{4}).
\]
Moreover, it is a routine matter to verify that $S_{(4, 6)}$ is isomorphic to a subdirect product of two copies
of $S_6$. So $\mathsf{V}(S_{(4, 6)})=\mathsf{V}(S_6)$.
By the main result of \cite{zrc} we deduce that $S_{(4, 5)}$, $S_{(4, 6)}$,
$S_{(4, 10)}$, $S_{(4, 18)}$, $S_{(4, 23)}$ and $S_{(4, 27)}$
are all finitely based.

It is easily verified that $S_{(4, 8)}$ is isomorphic to $S_c(ab)$.
By \cite[Proposition 3.2]{rjzl} we have that $S_{(4, 8)}$ is finitely based.
One can easily show that $S_{(4, 9)}$ is isomorphic to $S_c(a^3)$.
By \cite[Lemma 3.2]{wzr23} it follows that $S_{(4, 9)}$ is finitely based.

We know that $S_{(4, 29)}$ is isomorphic to the semiring $B_0$,
which is shown to be finitely based in \cite{shap23}.
So $S_{(4, 29)}$ is finitely based.

It is easy to see that $S_{(4, 37)}$ is the flat extension of an abelian group.
By \cite[Theorem 7.3]{jac08} we deduce that $S_{(4, 37)}$ is finitely based.
\end{proof}

Let $\bu=\bu_1+\bu_2+\cdots+\bu_n$ be an ai-semiring term such that $\bu_i\in X^+$ for all $1\leq i\leq n$,
and let $k\geq 1$.
Then $L_k(\bu)$ denotes the set of all $\bu_i$ in $\bu$ such that $\ell(\bu_i)=k$.
\begin{pro}\label{pro24070901}
$\mathsf{V}(S_{(4, 4)})$ is the ai-semiring variety determined by the identities
\begin{align}
x_1x_2x_3&\approx x_1x_2x_3+x_4;    \label{f0707001}\\
x^2 &\approx x^2+y;                 \label{f0707002}\\
x+xy &\approx x^2;                  \label{f0707003}\\
x+yx &\approx x^2;                   \label{f0707004}\\
x_1x_2+x_3x_4&\approx x_1x_2+x_3x_4+x_1x_4.    \label{f0707006}
\end{align}
\end{pro}
\begin{proof}
It is easy to verify that $S_{(4, 4)}$ satisfies the identities (\ref{f0707001})--(\ref{f0707006}).
In the remainder it is enough to show that every ai-semiring identity of $S_{(4, 4)}$
is derivable from (\ref{f0707001})--(\ref{f0707006}) and the identities defining $\mathbf{AI}$.
Let $\bu \approx \bu+\bq$ be such an identity, where $\bu=\bu_1+\bu_2+\cdots+\bu_n$ and $\bu_i, \bq\in X^+$, $1 \leq i \leq n$.
If $\ell(\bu_i)\geq 3$ for some $\bu_i\in \bu$, then
\[
\bu \approx \bu+\bu_i \stackrel{(\ref{f0707001})}\approx \bu+\bu_i+\bq.
\]
This derives the identity $\bu \approx \bu+\bq$.
Suppose that $c(L_1(\bu))\cap c(L_2(\bu)) \neq\emptyset$. Then
either $x, xy\in \bu$ or $x, yx\in \bu$ for some $x, y\in X$.
If $x, xy\in \bu$, then
\[
\bu \approx \bu+x+xy \stackrel{(\ref{f0707003})}\approx \bu+x^2\stackrel{(\ref{f0707002})}\approx \bu+x^2+\bq.
\]
Similarly, if $x, yx\in \bu$, then
\[
\bu \approx \bu+x+yx \stackrel{(\ref{f0707004})}\approx \bu+x^2\stackrel{(\ref{f0707002})}\approx \bu+x^2+\bq.
\]
So we can derive $\bu \approx \bu+\bq$.
If $h(L_2(\bu)) \cap t(L_2(\bu))\neq \emptyset$, then
$h(\bu_i)=t(\bu_j)$ for some $\bu_i, \bu_j \in L_2(\bu)$. Now we have
\begin{align*}
\bu
&\approx \bu+\bu_i+\bu_j\\
&\approx \bu+h(\bu_i)t(\bu_i)+h(\bu_j)t(\bu_j)&&\\
&\approx \bu+h(\bu_i)t(\bu_i)+h(\bu_j)t(\bu_j)+h(\bu_i)t(\bu_j)&&(\text{by}~(\ref{f0707006}))\\
&\approx \bu+h(\bu_i)t(\bu_i)+h(\bu_j)t(\bu_j)+h(\bu_i)h(\bu_i)\\
&\approx \bu+h(\bu_i)t(\bu_i)+h(\bu_j)t(\bu_j)+h(\bu_i)h(\bu_i)+\bq. &&(\text{by}~(\ref{f0707002}))
\end{align*}
This proves $\bu\approx \bu+\bq$.
Finally, assume that $\ell(\bu_i)\leq 2$ for all $\bu_i\in \bu$,
$c(L_1(\bu))\cap c(L_2(\bu)) =\emptyset$, and $h(L_2(\bu)) \cap t(L_2(\bu))=\emptyset$.
Consider the semiring homomorphism $\varphi: P_f(X^+)\to S_{(4, 4)}$ defined by the rule:
$\varphi(x)=2$ if $x \in c(L_1(\bu))$, $\varphi(x)=3$ if $x \in h(L_2(\bu))$,
$\varphi(x)=4$ if $x \in t(L_2(\bu))$, and $\varphi(x)=1$ otherwise.
It is easy to see that $\varphi(\bu)=2$ and so $\varphi(\bq)=2$.
This implies that $\ell(\bq)\leq 2$ and $c(\bq)\subseteq c(\bu)$.
If $\ell(\bq)=1$, then $\bq\in L_1(\bu)$ and so $\bu \approx \bu+\bq$ is trivial.
If $\ell(\bq)=2$ then $h(\bq)\in h(L_2(\bu))$ and $t(\bq)\in t(L_2(\bu))$.
This implies that $h(\bq)=h(\bu_i)$ and $t(\bq)=t(\bu_j)$ for some $\bu_i, \bu_j \in \bu$.
Furthermore, we have
\begin{align*}
\bu
&\approx \bu+\bu_i+\bu_j\\
&\approx \bu+h(\bu_i)t(\bu_i)+h(\bu_j)t(\bu_j)&&\\
&\approx \bu+h(\bu_i)t(\bu_i)+h(\bu_j)t(\bu_j)+h(\bu_i)t(\bu_j)&&(\text{by}~(\ref{f0707006}))\\
&\approx \bu+h(\bu_i)t(\bu_i)+h(\bu_j)t(\bu_j)+h(\bq)t(\bq)\\
&\approx \bu+h(\bu_i)t(\bu_i)+h(\bu_j)t(\bu_j)+\bq. &&
\end{align*}
This derives the identity $\bu\approx \bu+\bq$.
\end{proof}

\begin{remark}
Proposition \ref{pro24070901} shows that $S_{(4, 4)}$ is finitely based.
On the other hand, it is easy to see that $S_{(4, 4)}$ is isomorphic to $S(ab)$,
whose finite basis property has been established by Gao and Ren \cite{gr}, using an entirely unrelated proof.
\end{remark}

\begin{pro}
$\mathsf{V}(S_{(4, 14)})$ is the ai-semiring variety determined by the identities
\begin{align}
xy &\approx yx;                      \label{f07070100}\\
xy &\approx xy+x^2;                 \label{f0707011}\\
x+xy &\approx x^3;                \label{f0707012}\\
x_1x_2x_3&\approx x_1x_2x_3+x_4;    \label{f0707010}\\
xy+yz&\approx xy+yz+xz.            \label{f07070101}
\end{align}
\end{pro}
\begin{proof}
It is easy to verify that $S_{(4, 14)}$ satisfies the identities (\ref{f07070100})--(\ref{f07070101}).
In the remainder it is enough to show that every ai-semiring identity of $S_{(4, 14)}$
is derivable from (\ref{f07070100})--(\ref{f07070101}) and the identities defining $\mathbf{AI}$.
Let $\bu \approx \bu+\bq$ be such an identity, where $\bu=\bu_1+\bu_2+\cdots+\bu_n$.
Since (\ref{f07070100}) is satisfied by $S_{(4, 14)}$, we may assume that $\bu_i, \bq\in X_c^+$, $1 \leq i \leq n$.
If $\ell(\bu_i)\geq 3$ for some $\bu_i\in \bu$, then
\[
\bu \approx \bu+\bu_i \stackrel{(\ref{f0707010})}\approx \bu+\bu_i+\bq.
\]
This derives the identity $\bu \approx \bu+\bq$.
If $c(L_1(\bu))\cap c(L_2(\bu)) \neq\emptyset$,
then $x, xy\in \bu$ for some $x, y \in X$.
So we have
\[
\bu \approx \bu+x+xy \stackrel{(\ref{f0707012})}\approx \bu+x^3\stackrel{(\ref{f0707010})}\approx \bu+x^3+\bq.
\]
This implies $\bu \approx \bu+\bq$.
Now assume that $\ell(\bu_i)\leq 2$ for all $\bu_i\in \bu$
and that $c(L_1(\bu))\cap c(L_2(\bu))=\emptyset$.
Let $\varphi: P_f(X^+)\to S_{(4, 14)}$ be a semiring homomorphism defined by
$\varphi(x)=2$ if $x \in c(L_1(\bu))$, $\varphi(x)=3$ if $x \in c(L_2(\bu))$, and $\varphi(x)=1$ otherwise.
It is easy to see that $\varphi(\bu)=2$ and so $\varphi(\bq)=2$.
This implies that $\ell(\bq)\leq 2$ and $c(\bq)\subseteq c(\bu)$.
If $\ell(\bq)=1$, then $\bq\in L_1(\bu)$ and so $\bu \approx \bu+\bq$ is trivial.

Assume that $\ell(\bq)=2$. Then $c(\bq)\subseteq c(L_2(\bu))$. Let us write $\bq=xy$.
%Suppose that $\bu_i=x^2$ for all $\bu_i \in \bu$ with $ x \in c(\bu_i)$.
%Let $\psi: P_f(X^+)\to S_{(4, 14)}$ be a semiring homomorphism defined by
%$\psi(x)=3$, $\psi(z)=4$ if $z\in c(L_2(\bu))$ and $z\neq x$, $\psi(z)=2$ if $z \in c(L_1(\bu))$,
%and $\psi(z)=1$ otherwise. Then $\psi(\bu)=2$ and $\psi(\bq)=1$, a contradiction.
%So $\bu_i=xx_1$ for some $\bu_i \in \bu$ and some $x_1\in X\setminus\{x\}$.
%Similarly, one can prove that $\bu_j=yy_1$ for some $\bu_j \in \bu$ and some $y_1\in X\setminus\{y\}$.
%Let $L'_2(\bu)$ denote the set of all linear words in $L_2(\bu)$.
%Then $xx_1, yy_1 \in L'_2(\bu)$ and so $L'_2(\bu)\neq \emptyset$.
Now $L_2(\bu)$ can be thought of as a graph
whose vertex set is $c(L_2(\bu))$ and edge set consists of $\{x_2, y_2\}$ if $x_2y_2 \in L_2(\bu)$.
We have that both $x$ and $y$ are vertices of this graph.
Suppose by way of contradiction that there is no path connecting $x$ and $y$ in this graph.
Consider the semiring homomorphism $\psi: P_f(X^+)\to S_{(4, 14)}$ defined by
$\psi(z)=3$ if $z$ and $x$ are in the same connected component,
$\psi(z)=4$ for all remaining variables $z$ in $c(L_2(\bu))$,
and $\psi(z)=2$ if $z\in c(L_1(\bu))$.
Then $\psi(\bu)=2$ and $\psi(\bq)=1$, a contradiction.
It follows that there is a path connecting $x$ and $y$.
So we may assume that $xz_1, z_1z_2, \ldots, z_kz_{k+1}, z_{k+1}y\in \bu$
for some $z_1, z_2, \ldots, z_k, z_{k+1}\in X$.
We have
\[
\bu\approx \bu+xz_1+z_1z_2 \stackrel{(\ref{f07070101})}\approx \bu+xz_1+z_1z_2+xz_2.
\]
This implies the identity
\begin{equation}\label{f24071001}
\bu\approx \bu+xz_2.
\end{equation}
Furthermore, we obtain
\[
\bu\stackrel{(\ref{f24071001})}\approx \bu+xz_2\approx\bu+xz_2+z_2z_3\stackrel{(\ref{f07070101})}\approx\bu+xz_2+z_2z_3+xz_3.
\]
This proves the identity
\[
\bu\approx \bu+xz_3.
\]
We can continue this process and finally obtain
\[
\bu\approx \bu+xy.
\]
This derives $\bu\approx \bu+\bq$.
\end{proof}

In the remainder of this section we shall show that $S_{(4, 20)}$ is finitely based.
It is easy to see that $S_{(4, 20)}$ contains a copy of $S_{10}$,
which is the flat extension of a group of order two.
The following result, which follows from \cite[Corollary 2.13]{rz16} and its proof,
provides a solution of the equational problem for $S_{10}$.
Let $\bp$ be a word in $X^+$. Then $r(\bp)$ denotes the set
\[
\{x\in X \mid m(x, \bp)~\textrm{is an odd number}\}.
\]

\begin{lem}\label{lem24062901}
Let $\bu\approx \bu+\bq$ be an ai-semiring identity such that
$\bu=\bu_1+\bu_2+\cdots+\bu_n$, where $\bu_i, \bq \in X^+$, $1\leq i \leq n$.
Then $\bu\approx \bu+\bq$ is satisfied by $S_{10}$ if and only if
$c(\bq)\subseteq c(\bu)$ and $r(\bq)=r(\bu_{i_1}\bu_{i_2}\cdots \bu_{i_{3^k}})$
for some $\bu_{i_1}, \bu_{i_2}, \ldots, \bu_{i_{3^k}} \in \bu$.
\end{lem}

\begin{pro}
$\mathsf{V}(S_{(4, 20)})$ is the ai-semiring variety defined by the identities
\begin{align}
x^4& \approx x^2; \label{f24062901}\\
xy& \approx yx; \label{f24062902}\\
xy^2& \approx xy^2+x; \label{f24062903}\\
x_1x_2+x_3+x_4& \approx x_1x_2+x_3+x_4+x_1x_2x_3x_4. \label{f24062904}
\end{align}
\end{pro}
\begin{proof}
It is easy to verify that $S_{(4, 20)}$ satisfies the identities
(\ref{f24062901})--(\ref{f24062904}).
In the remainder we only need to show that every nontrivial ai-semiring identity of $S_{(4, 20)}$
is derivable from (\ref{f24062901})--(\ref{f24062904}) and the identities defining $\mathbf{AI}$.
Let $\bu \approx \bu+\bq$ be such an identity, where $\bu=\bu_1+\bu_2+\cdots+\bu_n$ and $\bu_i, \bq\in X^+_c$, $1 \leq i \leq n$.
Since $T_2$ is isomorphic to $\{1, 4\}$, we have that $T_2$ satisfies $\bu \approx \bu+\bq$ and so
$\ell(\bu_i)\geq 2$ for some $\bu_i \in \bu$.
On the other hand, it is easy to see that $S_{10}$ is isomorphic to $\{1, 3, 4\}$
and so $S_{10}$ satisfies $\bu \approx \bu+\bq$.
By Lemma \ref{lem24062901} we obtain that
$c(\bq)\subseteq c(\bu)$ and $r(\bq)=r(\bu_{i_1}\bu_{i_2}\cdots \bu_{i_{3^k}})$
for some $\bu_{i_1}, \bu_{i_2}, \ldots, \bu_{i_{3^k}} \in \bu$.
Let $c(\bq)=\{x_1, \ldots, x_r, y_1, \ldots, y_s\}$, where
$m(x_i, \bq)$ is an odd number and $m(y_j, \bq)$ is an even number
for all $1\leq i \leq r$ and $1\leq j \leq s$.
By the identities (\ref{f24062901}) and (\ref{f24062902}) we deduce
\begin{equation}\label{f24062910}
\bq \approx x_1^{k_1}x_2^{k_2}\cdots x_r^{k_r}y_1^2y_2^2\cdots y_s^2,
\end{equation}
where $k_i=1$ or $3$ for all $1\leq i \leq r$.

Now we have
\[
\bu\approx \bu+\bu_i+\bu_i+\bu_{i_1}\stackrel{(\ref{f24062904})}\approx \bu+\bu_i+\bu_i+\bu_{i_1}+\bu_i^2\bu_{i_1}.
\]
This proves the identity
\begin{equation}\label{f2407110001}
\bu\approx \bu+\bu_i^2\bu_{i_1}.
\end{equation}
Next, we can deduce
\[
\bu \stackrel{(\ref{f2407110001})}\approx \bu+\bu_i^2\bu_{i_1}\approx \bu+\bu_i+\bu_i^2\bu_{i_1}+\bu_{i_2}
\stackrel{(\ref{f24062904})}\approx \bu+\bu_i^2\bu_{i_1}\approx \bu+\bu_i+\bu_i^2\bu_{i_1}+\bu_{i_2}+\bu_i^3\bu_{i_1}\bu_{i_2}.
\]
This implies the identity
\[
\bu\approx \bu+\bu_i^3\bu_{i_1}\bu_{i_2}.
\]
Repeat this process and one can obtain
\begin{equation}\label{f2407110002}
\bu\approx \bu+\bu_i^{3^\ell+1}\bu_{i_1}\bu_{i_2}\cdots\bu_{i_{3^\ell}}.
\end{equation}
Furthermore, we have

\begin{align*}
\bu
&\approx \bu+\bu_i^{3^\ell+1}\bu_{i_1}^3\bu_{i_2}^3\cdots\bu_{i_{3^\ell}}^3&&(\text{by}~(\ref{f2407110002}))\\
&\approx \bu+\bu_i^{3^\ell+1}\bu_{i_1}^3\bu_{i_2}^3\cdots\bu_{i_{3^\ell}}^3\bu_1^2\bu_2^2\cdots \bu_n^2
&&  (\text{by}~  (\ref{f24062904}))\\
&\approx \bu+x_1^{k_1}x_2^{k_2}\cdots x_r^{k_r}y_1^2y_2^2\cdots y_s^2\bp^2&&(\text{by}~  (\ref{f24062901}), (\ref{f24062902}))\\
&\approx \bu+x_1^{k_1}x_2^{k_2}\cdots x_r^{k_r}y_1^2y_2^2\cdots y_s^2\bp^2+x_1^{k_1}x_2^{k_2}\cdots x_r^{k_r}y_1^2y_2^2\cdots y_s^2
&&  (\text{by}~  (\ref{f24062903}))\\
&\approx \bu+x_1^{k_1}x_2^{k_2}\cdots x_r^{k_r}y_1^2y_2^2\cdots y_s^2\bp^2+\bq.&&  (\text{by}~  (\ref{f24062910}))
\end{align*}
This derives $\bu\approx \bu+\bq$.
\end{proof}

\begin{remark}
It is a routine matter to verify that $S_{(4, 20)}$ is isomorphic to
a subdirect product of $S_{10}$ and $T_2$. So we have
\[
\mathsf{V}(S_{(4, 20)})=\mathsf{V}(S_{10}, T_2).
\]
\end{remark}

\section{Equational basis of $S_{(4, 15)}$, $S_{(4, 41)}$ and $S_{(4, 42)}$}
In this section we present finite equational basis for $S_{(4, 15)}$, $S_{(4, 41)}$ and $S_{(4, 42)}$.
It is easily verified that all of these semirings contain copies of $S_2$.
This requires that we be able to provide a solution of the equational problem for $S_2$.
\begin{lem}\label{lem24061201}
Let $\bu\approx \bu+\bq$ be an ai-semiring identity such that $\bu=\bu_1+\bu_2+\cdots+\bu_n$, where $\bu_i, \bq \in X^+$, $1\leq i \leq n$.
Then $\bu\approx \bu+\bq$ is satisfied by $S_2$ if and only if $\bu$ and $\bq$ satisfy one of the following conditions:
\begin{itemize}
\item[$(1)$] $\ell(\bu_i)\geq 3$ for some $\bu_i \in \bu$;

\item[$(2)$] $c(L_1(\bu))\cap c(L_2(\bu)) \neq\emptyset$;

\item[$(3)$] $\ell(\bu_i)\leq 2$ for all $\bu_i \in \bu$, $c(L_1(\bu))\cap c(L_2(\bu))=\emptyset$ and $\ell(\bq)\leq 2$.
             If $\ell(\bq)=1$, then $\bu\approx \bu+\bq$ is trivial. If $\ell(\bq)=2$, then $c(\bq)\subseteq c(L_2(\bu))$.
\end{itemize}
\end{lem}
\begin{proof}
Firstly, it is easy to verify that $S_2$ satisfies the following identities
\begin{align}
x_1x_2x_3 & \approx y_1y_2y_3; \label{f1}\\
xy        & \approx yx;    \label{f2}\\
xy+x& \approx x^3;         \label{f3}\\
x^3+y& \approx x^3;    \label{f4}\\
xy &\approx x^2+y^2.   \label{f5}
\end{align}
Suppose that $\ell(\bu_i)\geq 3$ for some $\bu_i \in \bu$.
Let $\varphi: P_f(X^+) \to S_2$ be an arbitrary substitution.
Since the identities (\ref{f1}) and (\ref{f4})
are satisfied by $S_2$, it follows immediately that $\varphi(\bu)=1$ and so
$\varphi(\bu)=\varphi(\bu)+\varphi(\bq)=\varphi(\bu+\bq)$.

Assume that $L_1(\bu)\cap L_2(\bu)\neq\emptyset$. Let $\varphi: P_f(X^+) \to S_2$ be an arbitrary substitution.
Since the identities (\ref{f1})--(\ref{f4})
are satisfied by $S_2$, we can deduce that $\varphi(\bu)=1$ and so
$\varphi(\bu)=\varphi(\bu+\bq)$.

Suppose that $\ell(\bu_i)\leq 2$ for all $\bu_i \in \bu$, $c(L_1(\bu))\cap c(L_2(\bu))=\emptyset$, $\ell(\bq)=2$ and $c(\bq)\subseteq c(L_2(\bu))$.
It is easy to see that the identity (\ref{f5}) implies the identity $\bu\approx \bu+\bq$.
Since $S_2$ satisfies (\ref{f5}), it follows immediately that $S_2$ also satisfies $\bu\approx \bu+\bq$.

Conversely, suppose that $S_2$ satisfies $\bu\approx \bu+\bq$, where $\ell(\bu_i)\leq 2$ for all $\bu_i \in \bu$,
$c(L_1(\bu))\cap c(L_2(\bu))=\emptyset$.
Let $\varphi: X \to S_2$ be a substitution
such that $\varphi(x)=2$ if $x\in c(L_1(\bu))$, $\varphi(x)=3$ if $x\in c(L_2(\bu))$,
and $\varphi(x)=1$ for all remaining variables $x$.
Then $\varphi(\bu)=2$ and so $\varphi(\bq)=2$.
This implies that $\ell(\bq)\leq 2$ and $c(\bq)\subseteq c(\bu)$.
If $\ell(\bq)=1$, then $\bq=x$ for some $x\in X$ and so $\varphi(x)=\varphi(\bq)=2$.
Thus $x\in c(L_1(\bu))$ and so $\bu\approx \bu+\bq$ is trivial.
If $\ell(\bq)=2$, we write $\bq=xy$ for some $x, y\in X$.
It follows that $\varphi(x)\varphi(y)=\varphi(xy)=\varphi(\bq)=2$
and so $\varphi(x)=\varphi(y)=3$.
Hence $x, y\in c(L_2(\bu))$ and so $c(\bq)\subseteq c(L_2(\bu))$.
This completes the proof.
\end{proof}

\begin{pro}
$\mathsf{V}(S_{(4, 15)})$ is the ai-semiring variety defined by the identities
\begin{align}
xy &\approx yx; \label{f061404}\\
xy &\approx x^2+y^2; \label{f061406}\\
xyz&\approx xyz+x; \label{f061402}\\
x^2+x&\approx x^3; \label{f061403}\\
x_1x_2x_3+x_4 & \approx x_1x_2x_3x_4. \label{f061401}
\end{align}
\end{pro}
\begin{proof}
It is easily verified that $S_{(4, 15)}$ satisfies the identities
(\ref{f061404})--(\ref{f061401}).
In the remainder we need only prove that every ai-semiring identity of $S_{(4, 15)}$
can be derived by (\ref{f061404})--(\ref{f061401}) and the identities defining $\mathbf{AI}$.
Let $\bu \approx \bu+\bq$ be such an identity, where $\bu=\bu_1+\bu_2+\cdots+\bu_n$
and $\bu_i, \bq\in X_c^+$, $1 \leq i \leq n$.
Since $M_2$ is isomorphic to $\{1, 4\}$, it follows that $M_2$ satisfies $\bu \approx \bu+\bq$ and so $c(\bq) \subseteq c(\bu)$.
Since $S_2$ is isomorphic to $\{1, 2, 3\}$, we have that $S_2$ also satisfies $\bu \approx \bu+\bq$.
By Lemma \ref{lem24061201} it is enough to consider the following three cases:

\textbf{Case 1.} $\ell(\bu_i) \geq 3$ for some $\bu_i \in \bu$. Then
\[
\bu \approx \bu+\bu_i \stackrel{(\ref{f061401}), (\ref{f061404})}\approx
\bu+\bu_i+\bu_1^k\bu _2^k \cdots \bu_n^k \stackrel{(\ref{f061404})}\approx \bu+\bq\bp
\stackrel{(\ref{f061402})}\approx \bu+\bq\bp+\bq,
\]
where $k>m(x, \bq)+1$ for all $x \in c(\bq)$ and $\ell(\bp)\geq 2$. This implies $\bu \approx \bu+\bq$.

\textbf{Case 2.} $c(L_1(\bu))\cap c(L_2(\bu))\neq\emptyset$. Take $x$ in $c(L_1(\bu))\cap c(L_2(\bu))$.
Then $x\in L_1(\bu)$ and $xy\in L_2(\bu)$ for some $y\in X$. So we have
\[
\bu \approx \bu+x+xy\stackrel{(\ref{f061406})}\approx \bu+x+x^2+y^2\stackrel{(\ref{f061403})}\approx \bu+y^2+x^3.
\]
This derives the identity $\bu \approx \bu+x^3$.
The remaining process are similar to the preceding case.

\textbf{Case 3.} $\ell(\bu_i)\leq 2$ for all $\bu_i \in u$, $c(L_1(\bu))\cap c(L_2(\bu))=\emptyset$,
$\ell(\bq)=2$ and $c(\bq)\subseteq c(L_2(\bu))$.
If we write $\bq=xy$ for some $x, y \in X$,
then $xx_1, yy_1\in L_2(\bu)$ for some $x_1, y_1\in X$. Furthermore, we have
\begin{align*}
\bu
&\approx \bu+xx_1+yy_1\\
&\approx \bu+x_1^2+y_1^2+x^2+y^2&&(\text{by}~(\ref{f061406}))\\
&\approx \bu+x_1^2+y_1^2+xy&&(\text{by}~(\ref{f061406}))\\
&\approx \bu+x_1^2+y_1^2+\bq.
\end{align*}
This proves $\bu \approx \bu+\bq$.
\end{proof}

\begin{remark}
It is easy to see that $S_{(4, 15)}$ is an idempotent extension of $S_2$.
So $\mathsf{V}(S_{(4, 15)})= \mathsf{V}(S_2, M_2)$.
This shows that $\mathsf{V}(S_{(4, 15)})$ is the regularization of $\mathsf{V}(S_2)$.
\end{remark}

\begin{pro}\label{pro24061311}
$\mathsf{V}(S_{(4, 41)})$ is the ai-semiring variety defined by the identities
\begin{align}
xy+x& \approx xy+x^3;         \label{f11}\\
yx+x& \approx yx+x^3;    \label{f10}\\
x_1y_1z_1+x_2y_2 & \approx x_1y_1z_1+x_2; \label{f06}\\
x_1y_1z_1+x_2& \approx x_1y_1z_1+x_2+x_2y_2;    \label{f7}\\
x_1y_1+x_2y_2&\approx x_1y_1+x_2y_2+x_1x_2; \label{f8}\\
x_1y_1+x_2y_2&\approx x_1y_1+x_2y_2+x_1y_2. \label{f9}
\end{align}
\end{pro}
\begin{proof}
It is easy to verify that $S_{(4, 41)}$ satisfies the identities (\ref{f11})--(\ref{f9}).
In the remainder it is enough to prove that every ai-semiring identity of $S_{(4, 41)}$
can be derived by (\ref{f11})--(\ref{f9}) and the identities defining $\mathbf{AI}$.
Let $\bu \approx \bu+\bq$ be such an identity, where $\bu=\bu_1+\bu_2+\cdots+\bu_n$ and $\bu_i, \bq\in X^+$, $1 \leq i \leq n$.
It is easy to see that $L_2$ is isomorphic to $\{1, 4\}$ and so $L_2$ satisfies  $\bu \approx \bu+\bq$.
This implies that there exists $\bu_i \in \bu$ such that $h(\bu_i)=h(\bq)$ and so $\bu_i=h(\bq)s(\bu_i)$.
On the other hand, we have that $S_2$ is isomorphic to $\{1, 2, 3\}$ and so it satisfies $\bu \approx \bu+\bq$.
By Lemma \ref{lem24061201} we only need to consider the following three cases:

\textbf{Case 1.} $\ell(\bu_j)\geq 3$ for some $\bu_j\in \bu$. Then
\begin{align*}
\bu
&\approx \bu+\bu_j+\bu_i\\
&\approx  \bu+\bu_j+h(\bq)s(\bu_i)\\
&\approx  \bu+\bu_j+h(\bq)      &&(\text{by}~(\ref{f06}))\\
&\approx  \bu+\bu_j+h(\bq)+h(\bq)s(\bq)  &&(\text{by}~(\ref{f7}))\\
&\approx  \bu+\bu_j+h(\bq)+\bq.
\end{align*}
This derives $\bu \approx \bu+\bq$.

\textbf{Case 2.} $c(L_1(\bu))\cap c(L_2(\bu))\neq\emptyset$. Take $x$ in $c(L_1(\bu))\cap c(L_2(\bu))$.
Then $x\in L_1(\bu)$.
If $xy\in L_2(\bu)$ for some $y \in X$, then
\[
\bu\approx \bu+xy+x \stackrel{(\ref{f11})} \approx \bu+xy+x^3.
\]
If $yx\in L_2(\bu)$ for some $y \in X$, then
\[
\bu\approx \bu+yx+x \stackrel{(\ref{f10})} \approx \bu+yx+x^3.
\]
So we can derive $\bu\approx \bu+x^3$.
The remaining steps are similar to the preceding case.

\textbf{Case 3.} $\ell(\bu_k)\leq 2$ for all $\bu_k \in u$, $c(L_1(\bu))\cap c(L_2(\bu))=\emptyset$,
$\ell(\bq)=2$ and $c(\bq)\subseteq c(L_2(\bu))$.
Then $\bu_i \in L_2(\bu)$.
We may assume that $\bq=xy$ for some $x, y \in X$ and so $\bu_i=xx_1$ for some $x_1 \in X$.
If $yy_1\in L_2(\bu)$ for some $y_1\in X$, then
\[
\bu\approx \bu+xx_1+yy_1 \stackrel{(\ref{f8})}{\approx}  \bu+xx_1+yy_1+xy \approx \bu+xx_1+yy_1+\bq.
\]
If $y_2y\in L_2(u)$ for some $y_2\in X$, then
\[
\bu\approx \bu+xx_1+y_2y \stackrel{(\ref{f9})}{\approx}  \bu+xx_1+yy_1+xy \approx \bu+xx_1+yy_1+\bq.
\]
So we obtain the identity $\bu \approx \bu+\bq$.
\end{proof}

\begin{remark}
It is a routine matter to verify that $S_{(4, 41)}$ is isomorphic to a subdirect product of $S_2$ and $S_5$.
So $\mathsf{V}(S_{(4, 41)})=\mathsf{V}(S_2, S_5)$.
On the other hand, it is easy to see that both $L_2$ and $T_2$ can be embedded into $S_5$.
Also, $S_5$ satisfies an equational basis of $\mathsf{V}(L_2, T_2)$ that can be found in \cite{sr}.
It follows that $\mathsf{V}(S_5)=\mathsf{V}(L_2, T_2)$ and so $\mathsf{V}(S_{(4, 41)})=\mathsf{V}(S_2, L_2, T_2)$.
Since $T_2$ can be embedded into $S_2$, we therefore have
\[
\mathsf{V}(S_{(4, 41)})=\mathsf{V}(S_2, L_2).
\]
\end{remark}

\begin{cor}
The ai-semiring $S_{(4, 16)}$ is finitely based.
\end{cor}
\begin{proof}
It is easy to see that $S_{(4, 16)}$ and $S_{(4, 41)}$ have dual multiplications.
By Proposition \ref{pro24061311} we immediately deduce that $S_{(4, 16)}$ is finitely based.
\end{proof}

\begin{pro}
$\mathsf{V}(S_{(4, 42)})$ is the ai-semiring variety defined by the identities
\begin{align}
x^4 & \approx x^3; \label{f061410} \\
xy  & \approx yx; \label{f061411} \\
x^3& \approx x+xy; \label{f061412} \\
x^2+yz&\approx x^2+yz+xy; \label{f061413}\\
x_1x_2x_3+x_4 & \approx x_1x_2x_3+x_4+x_4x_5. \label{f061414}
\end{align}
\end{pro}
\begin{proof}
It is easy to verify that $S_{(4, 42)}$ satisfies the identities (\ref{f061410})--(\ref{f061414}).
In the remainder it is enough to show that every ai-semiring identity of $S_{(4, 42)}$
is derivable from (\ref{f061410})--(\ref{f061414}) and the identities defining $\mathbf{AI}$.
Let $\bu \approx \bu+\bq$ be such an identity, where $\bu=\bu_1+\bu_2+\cdots+\bu_n$
and $\bu_i, \bq\in X_c^+$, $1 \leq i \leq n$.
It is easy to see that $D_2$ is isomorphic to $\{1, 4\}$
and so $D_2$ satisfies $\bu \approx \bu+\bq$. This implies that $c(\bu_i) \subseteq c(\bq)$ for some $\bu_i\in \bu$.
Let $c(\bu_i)=\{x_1, x_2, \ldots, x_m\}$ and $\bq=x_1x_2\cdots x_m\bq'$ for some word $\bq'$.
Since $S_2$ is isomorphic to $\{1, 2, 3\}$, it follows that $S_2$ also satisfies $\bu \approx \bu+\bq$.
By Lemma \ref{lem24061201} it suffices to consider the following three cases:

\textbf{Case 1.} $\ell(\bu_j) \geq 3$ for some $\bu_j \in \bu$. We have
\begin{align*}
\bu
&\approx \bu+\bu_j+\bu_i\\
&\approx \bu+\bu_j+(x_1x_2\cdots x_m)^3&&(\text{by}~(\ref{f061410}), (\ref{f061411}), (\ref{f061414}))\\
&\approx \bu+\bu_j+x_1x_2\cdots x_m+x_1x_2\cdots x_m\bq'&&(\text{by}~(\ref{f061412}))\\
&\approx \bu+\bu_j+x_1x_2\cdots x_m+\bq.
\end{align*}
This derives $\bu\approx \bu+\bq$.

\textbf{Case 2.} $c(L_1(\bu))\cap c(L_2(\bu)) \neq\emptyset$. Let $x$ be a variable in $c(L_1(\bu))\cap c(L_2(\bu))$.
Then $xy\in L_2(\bu)$ for some $y\in X^+$ and so we have
\[
\bu \approx u+x+xy \stackrel{(\ref{f061412})}\approx \bu+x^3.
\]
The remaining steps are similar to the preceding case.

\textbf{Case 3.} $\ell(\bu_k)\leq 2$ for all $\bu_k \in u$,
$c(L_1(\bu))\cap c(L_2(\bu))=\emptyset$, $\ell(\bq)=2$ and $c(\bq)\subseteq c(L_2(\bu))$.
Since $c(\bu_i) \subseteq c(\bq)$, it follows that $c(\bu_i)\subseteq c(L_2(\bu))$ and so $\bu_i\in L_2(\bu)$.
Let $\bq=xy$ for some $x, y \in X$.
Then $\bu_i=xy, x^2$ or $y^2$.
If $\bu_i=xy$, then $\bu\approx \bu+\bq$ is trivial.
If $\bu_i=x^2$, we may assume that $yy_1 \in L_2(\bu)$ for some $y_1 \in X^+$. Then
\[
\bu\approx \bu+x^2+yy_1 \stackrel{(\ref{f061413})}\approx \bu+x^2+yy_1+xy \approx \bu+\bq.
\]
This derives $\bu\approx \bu+\bq$.
If $\bu_i=y^2$, this is similar to the preceding case.
\end{proof}

\begin{remark}
It is easy to verify that $S_{(4, 42)}$ is isomorphic to a subdirect product of $S_2$ and $S_{13}$.
So $\mathsf{V}(S_{(4, 42)})=\mathsf{V}(S_2, S_{13})$.
On the other hand, it is easy to see that both $D_2$ and $T_2$ can be embedded into $S_{13}$.
Also, $S_{13}$ satisfies an equational basis of $\mathsf{V}(D_2, T_2)$ that can be found in \cite{sr}.
Thus $\mathsf{V}(S_{13})=\mathsf{V}(D_2, T_2)$ and so $\mathsf{V}(S_{(4, 42)})=\mathsf{V}(S_2, D_2, T_2)$.
Since $T_2$ can be embedded into $S_2$, we therefore obtain
\[
\mathsf{V}(S_{(4, 42)})=\mathsf{V}(S_2, D_2).
\]
\end{remark}

\section{Equational basis of $S_{(4, 30)}$, $S_{(4, 47)}$ and $S_{(4, 48)}$}
In this section we give finite equational basis for  $S_{(4, 30)}$, $S_{(4, 47)}$ and $S_{(4, 48)}$.
It is easily verified that all of these semirings contain copies of $S_4$.
So we first provide a solution of the equational problem for $S_4$.
\begin{lem}\label{lem24061501}
Let $\bu\approx \bu+\bq$ be a nontrivial ai-semiring identity such that
$\bu=\bu_1+\bu_2+\cdots+\bu_n$, where $\bu_i, \bq \in X^+$, $1\leq i \leq n$.
Then $\bu\approx \bu+\bq$ is satisfied by $S_4$ if and only if $\bu$ and $\bq$ satisfy the following conditions:
\begin{itemize}
\item[$(1)$] $c(\bq)\subseteq  c(\bu)$;

\item[$(2)$] $\ell(\bu_i)\geq 2$ for some $\bu_i \in \bu$;

\item[$(3)$] If $\bu$ satisfies the property {\rm (T)}:
\[
(\forall \bu_i, \bu_j \in \bu) ~ m(t(\bu_i), \bu_j)\leq 1; m(t(\bu_i), \bu_j)= 1\Rightarrow t(\bu_i)=t(\bu_j),
\]
then $\bu+\bq$ also satisfies the property {\rm (T)}, that is,
\[
(\forall \bu_i \in \bu) ~ m(t(\bu_i), \bq)\leq 1; m(t(\bu_i), \bq)= 1\Rightarrow t(\bu_i)=t(\bq);
\]
\[
m(t(\bq), \bq)= 1; m(t(\bq), \bu_i)\leq 1; m(t(\bq), \bu_i)= 1\Rightarrow t(\bq)=t(\bu_i).
\]
\end{itemize}
\end{lem}
\begin{proof}
Suppose that $\bu\approx \bu+\bq$ is satisfied by $S_4$.
Since $M_2$ and $T_2$ can be embedded into $S_4$, it follows that both $M_2$ and $T_2$ satisfy $\bu \approx \bu+\bq$
and so $c(\bq)\subseteq  c(\bu)$ and $\ell(\bu_i)\geq 2$ for some $\bu_i \in \bu$.
Assume that $\bu$ satisfies the property {\rm (T)}.
Let $\varphi: X \to S_4$ be a substitution such that $\varphi(x)=2$ if $x\in t(\bu)$
and $\varphi(x)=3$ otherwise. It is easy to see that $\varphi(\bu)=2$ and so $\varphi(\bq)=2$.
This implies that $\varphi(t(\bq))=2$ and $\varphi(x)=3$ for all $x\in c(p(\bq))$.
Thus
$t(\bq)\in t(\bu)$, $m(t(\bq), \bq)=1$ and $c(p(\bq))\cap t(\bu)=\emptyset$.
Let $\bu_i$ be an arbitrary word in $\bu$.
Now it follows that $m(t(\bu_i), \bq)\leq 1$ and $m(t(\bu_i), \bq)=1$ implies that $t(\bu_i)=t(\bq)$.
Since $t(\bq)\in t(\bu)$ and $\bu$ satisfies the property {\rm (T)}, we deduce that
$m(t(\bq), \bu_i)\leq 1$ and  $m(t(\bq), \bu_i)= 1$ implies that $t(\bq)=t(\bu_i)$.
So we have shown that $\bu+\bq$ satisfies the property {\rm (T)}.

Conversely, assume that $\bu$ and $\bq$ satisfy the conditions $(1)$, $(2)$ and $(3)$.
Then
$c(\bq)\subseteq  c(\bu)$ and $\ell(\bu_i)\geq 2$ for some $\bu_i \in \bu$.
Let $\varphi: X \to S_4$ be an arbitrary substitution.
If $\varphi(\bu)=1$, then
\[
\varphi(\bu+\bq)=\varphi(\bu)+\varphi(\bq)=1+\varphi(\bq)=1.
\]
If $\varphi(\bu)=3$, then $\varphi(x)=3$ for all $x\in c(\bu)$ and so $\varphi(\bq)=3$.
This shows that $\varphi(\bu+\bq)=\varphi(\bu)=3$.
If $\varphi(\bu)=2$, then $\varphi(t(\bu_i))=2$ for all $\bu_i \in \bu$ and $\varphi(x)=3$ if $x\in c(p(\bu_i))$ for some $\bu_i \in u$.
This shows that $\bu$ satisfies the property {\rm (T)}.
By assumption we have that $\bu+\bq$ also satisfies the property {\rm (T)}.
So $m(t(\bq), \bq)=1$ and $c(p(\bq))\cap t(\bu)=\emptyset$.
Since $c(\bq) \subseteq c(\bu)$, it follows that $t(\bq) \in c(\bu)$
and so $t(\bq)=t(\bu_k)$ for some $\bu_k \in \bu$.
This implies that $\varphi(t(\bq))=2$.
Let $x\in c(p(\bq))$. Since $c(p(\bq))\cap t(\bu)=\emptyset$,
we have that $x\in c(p(\bu_\ell))$ for some $\bu_\ell \in \bu$ and so $\varphi(x)=3$.
Thus $\varphi(\bq)=2$ and so
\[
\varphi(\bu+\bq)=\varphi(\bu)+\varphi(\bq)=2+2=2=\varphi(\bu).
\]
We conclude that $\bu\approx \bu+\bq$ is satisfied by $S_4$.
\end{proof}

\begin{remark}
We present an equivalent statement of Lemma \ref{lem24061501} by a notation that is introduced by \cite[Proposition 5.5]{jrz}.
Let $\bv=\bv_1+\bv_2+\cdots+\bv_m$ be an ai-semiring term such that $\bv_j\in X^+$ for all $1\leq j \leq m$.
Then $\delta(\bv)$ denotes the set of nonempty subsets $Z$ of $c(\bv)$ that satisfy the following two conditions:
\begin{itemize}
\item $Z \cap c(\bv_i)$ is a singleton for every $\bv_i\in\bv$;
\item $m(x, \bv_i ) =1$ if $\{x\}=Z \cap c(\bv_i)$.
\end{itemize}
One can observe that $t(\bv)\in \delta(\bv)$ if and only if $\bv$ satisfies the property {\rm (T)}.
By Lemma \ref{lem24061501} we immediately deduce that a nontrivial ai-semiring identity $\bu\approx \bu+\bq$ is satisfied by $S_4$ if and only if $\bu$ and $\bq$ satisfy the following conditions:
\begin{itemize}
\item[$(1)$] $c(\bq)\subseteq  c(\bu)$;

\item[$(2)$] $\ell(\bu_i)\geq 2$ for some $\bu_i \in \bu$;

\item[$(3)$] If $t(\bu)\in \delta(\bu)$, then $t(\bu+\bq)\in \delta(\bu+\bq)$.
\end{itemize}
\end{remark}

\begin{pro}\label{pro24071030}
$\mathsf{V}(S_{(4, 30)})$ is the ai-semiring variety defined by the identities
\begin{align}
x^2y& \approx xy; \label{f24061801}\\
xyz& \approx yxz; \label{f24061802}\\
x+y^2& \approx x+y^2+y^2x^2;    \label{f24061803}\\
x+yz&\approx x+yz+yx; \label{f24061804}\\
xy    & \approx xy+y. \label{f24061807}
\end{align}
\end{pro}
\begin{proof}
It is easy to verify that $S_{(4, 30)}$ satisfies the identities
(\ref{f24061801})--(\ref{f24061807}).
In the remainder we only need to show that every ai-semiring identity of $S_{(4, 30)}$
is derivable from (\ref{f24061801})--(\ref{f24061807}) and the identities defining $\mathbf{AI}$.
Let $\bu \approx \bu+\bq$ be such an identity,
where $\bu=\bu_1+\bu_2+\cdots+\bu_n$ and $\bu_i, \bq\in X^+$, $1 \leq i \leq n$.
Since $R_2$ is isomorphic to $\{1, 3\}$, we have that $R_2$ satisfies $\bu \approx \bu+\bq$ and so
$t(\bq)=t(\bu_i)$ for some $\bu_i \in \bu$.
If $\ell(\bq)=1$, then $\bq=t(\bq)$ and so $\bq=t(\bu_i)$.
Furthermore, we have
\[
\bu \approx \bu+\bu_i \approx \bu+p(\bu_i)t(\bu_i)\approx \bu+p(\bu_i)\bq\stackrel{(\ref{f24061807})}\approx \bu+p(\bu_i)\bq+\bq.
\]
This derives $\bu \approx \bu+\bq$.
Next, suppose that $\ell(\bq) \geq 2$. Then $p(\bq)$ is a nonempty word.
It is easy to see that $S_4$ is isomorphic to $\{1, 2, 4\}$
and so $S_4$ satisfies $\bu \approx \bu+\bq$.
By Lemma \ref{lem24061501} we have that $c(\bq)\subseteq  c(\bu)$,
$\ell(\bu_j)\geq 2$ for some $\bu_j \in \bu$, and $\bu$ satisfies the property {\rm (T)}
implies that $\bu+\bq$ satisfies the property {\rm (T)}.

\textbf{Case 1.} $\bu$ satisfies the property {\rm (T)}.
Then $\bu+\bq$ also satisfies the property {\rm (T)}.
This implies that $m(t(\bq), \bq)=1$ and so $t(\bq)\notin c(p(\bq))$.
By the identities (\ref{f24061801}) and (\ref{f24061802}) we derive
\begin{equation}\label{f240714001}
\bq \approx x_1x_2\cdots x_{k}t(\bq),
\end{equation}
where $\{x_1, x_2, \ldots, x_k\}=c(p(\bq))$.
We may suppose that
$\{x_1, \ldots, x_r\}$ and $\{y_1, \ldots, y_s\}$ are two disjoint subsets of $c(\bu)$ such that
\[
c(\bu)=\{x_1, \ldots, x_r\} \bigcup \{y_1, \ldots, y_s\},
\]
and
\[
t(\bu)=\{y_1, \ldots, y_s\},
\]
where $y_1=t(\bq)$.
Now we have
\begin{align*}
\bu
&\approx \bu+x_{k+1}x_{k+2}\cdots x_rx_1x_2\cdots x_kt(\bq)&&(\text{by}~(\ref{f24061801}), (\ref{f24061802}), (\ref{f24061804}))\\
&\approx \bu+x_{k+1}x_{k+2}\cdots x_r \bq &&  (\text{by}~  (\ref{f240714001}))                   \\
&\approx \bu+x_{k+1}x_{k+2}\cdots x_r \bq+\bq.   &&  (\text{by}~  (\ref{f24061807}))
\end{align*}
This implies the identity $\bu \approx \bu+\bq$.

\textbf{Case 2.} $\bu$ does not satisfy the property {\rm (T)}.
Consider the following two subcases.

\textbf{Subcase 2.1.} $m(t(\bu_{i_1}), \bu_{i_2}) \geq 2$ for some $\bu_{i_1}, \bu_{i_2} \in \bu$.
Then $t(\bu_{i_1}) \in c(p(\bu_{i_2}))$. We have
\begin{align*}
\bu
&\approx \bu+\bu_{i_1}+\bu_{i_2} \\
&\approx \bu+\bu_{i_1}+p(\bu_{i_2})t(\bu_{i_2})\\
&\approx \bu+\bu_{i_1}+p(\bu_{i_2})t(\bu_{i_2})+p(\bu_{i_2})\bu_{i_1} &&(\text{by}~(\ref{f24061804}))\\
&\approx \bu+\bu_{i_1}+p(\bu_{i_2})t(\bu_{i_2})+(p(\bu_{i_2})\bu_{i_1})^2 &&(\text{by}~(\ref{f24061801}), (\ref{f24061802}))\\
&\approx \bu+\bu_{i_1}+p(\bu_{i_2})t(\bu_{i_2})+\bp^2\bq^2&&(\text{by}~(\ref{f24061802}), (\ref{f24061803}))\\
&\approx \bu+\bu_{i_1}+p(\bu_{i_2})t(\bu_{i_2})+\bp^2\bq^2+\bq.&&(\text{by}~  (\ref{f24061807}))
\end{align*}
So we obtain $\bu \approx \bu+\bq$.

\textbf{Subcase 2.2.} $m(t(\bu_{i_1}), \bu_{i_2}) = 1$ and $t(\bu_{i_1})\neq t(\bu_{i_2})$ for some $\bu_{i_1}, \bu_{i_2} \in \bu$.
This is similar to the preceding case.
\end{proof}

\begin{remark}
It is easily verified that $S_{(4, 30)}$ is isomorphic to a subdirect product of $S_4$ and $S_{14}$.
So $\mathsf{V}(S_{(4, 30)})=\mathsf{V}(S_4, S_{14})$. On the other hand, one can check that
both $R_2$ and $M_2$ can be embedded into $S_{14}$ and that $S_{14}$ satisfies an equational basis of
$\mathsf{V}(R_2, M_2)$, which can be found in \cite{sr}. This implies that
$\mathsf{V}(S_{(4, 14)})=\mathsf{V}(R_2, M_2)$. Notice that $M_2$ can be embedded into $S_4$.
We immediately deduce that
\[
\mathsf{V}(S_{(4, 30)})=\mathsf{V}(S_4, R_2).
\]
\end{remark}

\begin{cor}
$S_{(4, 45)}$ is finitely based.
\end{cor}
\begin{proof}
It is easy to see that $S_{(4, 45)}$ is isomorphic to the dual of $S_{(4, 30)}$.
By Proposition \ref{pro24071030} we immediately deduce that $S_{(4, 45)}$ is finitely based.
\end{proof}

\begin{pro}\label{pro24061701}
$\mathsf{V}(S_{(4, 47)})$ is the ai-semiring variety defined by the identities
\begin{align}
x^2y& \approx xy; \label{f24061549}\\
x_1x_2x_3x_4& \approx x_1x_3x_2x_4; \label{f24061550}\\
x^2&\approx x^2+x;\label{f24061553}\\
x^2y^2& \approx x^2y^2+x^2;  \label{f24061554}\\
x+y^2& \approx x+y^2+x^2y^2;    \label{f24061551}\\
x+y^2& \approx x+y^2+y^2x^2;    \label{f240615510}\\
x+yz&\approx x+yz+yx; \label{f24061552}\\
xy+zx & \approx zy+x^2.  \label{f24061554111}
\end{align}
\end{pro}
\begin{proof}
It is easy to verify that $S_{(4, 47)}$ satisfies the identities
(\ref{f24061549})--(\ref{f24061554111}).
In the remainder it is enough to show that every ai-semiring identity of $S_{(4, 47)}$
is derivable from (\ref{f24061549})--(\ref{f24061554111}) and the identities defining $\mathbf{AI}$.
Let $\bu \approx \bu+\bq$ be such an identity, where $\bu=\bu_1+\bu_2+\cdots+\bu_n$ and $\bu_i, \bq\in X^+$, $1 \leq i \leq n$.
Since $L_2$ is isomorphic to $\{1, 4\}$, it follows that $L_2$ satisfies $\bu \approx \bu+\bq$ and so
$h(\bq)=h(\bu_i)$ for some $\bu_i \in \bu$.
On the other hand, it is easy to see that $S_4$ is isomorphic to $\{1, 2, 3\}$
and so $S_4$ satisfies $\bu \approx \bu+\bq$.
By Lemma \ref{lem24061501} we have that $c(\bq)\subseteq  c(\bu)$,
$\ell(\bu_r)\geq 2$ for some $\bu_r \in \bu$, and $\bu$ satisfies the property {\rm (T)}
implies that $\bu+\bq$ satisfies the property {\rm (T)}.

\textbf{Case 1.} $\bu$ satisfies the property {\rm (T)}.
Then $\bu+\bq$ also satisfies the property {\rm (T)}.
So we have that
\[
(\forall \bu_j \in \bu) ~ m(t(\bu_j), \bq)\leq 1; m(t(\bu_j), \bq)= 1\Rightarrow t(\bu_j)=t(\bq);
\]
\[
m(t(\bq), \bq)= 1; m(t(\bq), \bu_j)\leq 1; m(t(\bq), \bu_j)= 1\Rightarrow t(\bq)=t(\bu_j).
\]
Recall that $\bu_i$ is a word in $\bu$ such that $h(\bu_i)=h(\bq)$.
We shall consider the following two subcases.

\textbf{Subcase 1.1.}
$\ell(\bu_i)=1$. Then $t(\bu_i)=h(\bu_i)$ and so $t(\bu_i)\in c(\bq)$.
It follows that $t(\bu_i)=t(\bq)$ and so $t(\bq)=h(\bq)$.
Thus $\ell(\bq)=1$ and so $\bq=\bu_i$.
Hence $\bu \approx \bu+\bq$ is trivial.

\textbf{Subcase 1.2.}
$\ell(\bu_i)\geq 2$. Then $p(\bu_i)$ is a nonempty word.
Since $c(\bq)\subseteq c(\bu)$,
there exists $\bu_r\in \bu$ such that $t(\bq)\in c(\bu_r)$ and so $t(\bq)=t(\bu_r)$.

If $\ell(\bq)=1$, then $\bq=t(\bq)=h(\bq)$. Furthermore, we have
\begin{align*}
\bu
&\approx \bu+\bu_i+\bu_r\\
&\approx \bu+h(\bu_i)s(\bu_i)+p(\bu_r)t(\bu_r)\\
&\approx \bu+\bq s(\bu_i)+p(\bu_r)\bq  \\
&\approx  \bu+\bq s(\bu_i)+p(\bu_r)\bq +\bq^2&&(\text{by}~(\ref{f24061552}), (\ref{f24061554111}))\\
&\approx \bu+\bq s(\bu_i)+p(\bu_r)\bq +\bq^2+\bq.  &&(\text{by}~(\ref{f24061553}))
\end{align*}
This implies the identity $\bu \approx \bu+\bq$.

Now suppose that $\ell(\bq)\geq 2$. Then $h(\bq) \neq t(\bq)$.
Since $\bu$ satisfies the property {\rm (T)}, we may suppose that
$\{x_1, \ldots, x_m\}$ and $\{y_1, \ldots, y_s\}$ are two disjoint subsets of $c(\bu)$ such that
\[
c(\bu)=\{x_1, \ldots, x_m\} \bigcup \{y_1, \ldots, y_s\},
\]
and
\[
t(\bu)=\{y_1, \ldots, y_s\},
\]
where $x_1=h(\bq)$ and $y_1=t(\bq)$.
Then
\begin{align*}
\bu
&\approx p(\bu_i)t(\bu_i)+\bu\\
&\approx \bu+x_1^2\cdots x_m^2y_1&&(\text{by}~(\ref{f24061549}), (\ref{f24061550}), (\ref{f24061552}))\\
&\approx \bu+(p(\bq))^2x_1^2\cdots x_m^2y_1                    &&(\text{by}~(\ref{f24061549}), (\ref{f24061550}))\\
&\approx \bu+(p(\bq))^2(x_1\cdots x_m)^2y_1                    &&(\text{by}~(\ref{f24061550}))                  \\
&\approx \bu+(p(\bq))^2(x_1\cdots x_m)^2y_1+(p(\bq))^2y_1      &&(\text{by}~(\ref{f24061554}))\\
&\approx \bu+(p(\bq))^2(x_1\cdots x_m)^2y_1+(p(\bq))^2y_1+ p(\bq)y_1           &&(\text{by}~(\ref{f24061553}))\\
&\approx \bu+(p(\bq))^2(x_1\cdots x_m)^2y_1+(p(\bq))^2y_1+\bq.
\end{align*}
This derives $\bu\approx \bu+\bq$.

\textbf{Case 2.} $\bu$ does not satisfy the property {\rm (T)}.
Consider the following two subcases:

\textbf{Subcase 2.1.} $m(t(\bu_j), \bu_k)\geq 2$ for some $\bu_j, \bu_k \in \bu$. Then $t(\bu_j) \in c(p(\bu_k))$ and so
\begin{align*}
\bu
&\approx \bu+\bu_j+\bu_k\\
&\approx \bu+\bu_j+p(\bu_k)t(\bu_k)\\
&\approx \bu+\bu_j+p(\bu_k)t(\bu_k)+p(\bu_k)\bu_j.&&(\text{by}~(\ref{f24061552}))
\end{align*}
This derives
\begin{equation}\label{f240714100}
\bu \approx \bu+p(\bu_k)\bu_j.
\end{equation}
Furthermore, we have
\begin{align*}
\bu
&\approx \bu+p(\bu_k)\bu_j && (\text{by}~(\ref{f240714100})) \\
&\approx \bu+p(\bu_k)p(\bu_j)t(\bu_j) \\
&\approx \bu+p(\bu_k)^2p(\bu_j)^2t(\bu_j)^2&&(\text{by}~(\ref{f24061549}), (\ref{f24061550}))\\
&\approx \bu+(p(\bu_k)\bu_j)^2 &&(\text{by}~(\ref{f24061550}))\\
&\approx \bu+(h(\bq)x_1\cdots x_m)^2 &&(\text{by}~(\ref{f24061551}), (\ref{f240615510}))\\
&\approx \bu+\bq^2\bp^2 &&(\text{by}~(\ref{f24061550}))\\
&\approx \bu+\bq^2\bp^2+\bq^2&&(\text{by}~(\ref{f24061554}))\\
&\approx \bu+\bq^2\bp^2+\bq^2+\bq, &&(\text{by}~(\ref{f24061553}))
\end{align*}
where $c(\bu)=\{h(\bq), x_1, \ldots, x_m\}$.
This implies the identity $\bu \approx \bu+\bq$.

\textbf{Subcase 2.2.} $m(t(\bu_j), \bu_k)=1$ and $t(\bu_j) \neq t(\bu_k)$ for some $\bu_j, \bu_k \in \bu$.
Then $m(t(\bu_j), p(\bu_k))=1$ and so $m(t(\bu_j), p(\bu_k)p(\bu_k)t(\bu_k))=2$. Moreover, we have
\[
\bu \approx \bu+\bu_k \approx \bu+p(\bu_k)t(\bu_k) \stackrel{(\ref{f24061549})}\approx \bu+p(\bu_k)p(\bu_k)t(\bu_k).
\]
The remaining steps are similar to the preceding case.
\end{proof}

\begin{remark}
It is easy to verify that $S_{(4, 47)}$ is isomorphic to a subdirect product of $S_4$ and $S_9$.
So $\mathsf{V}(S_{(4, 47)})=\mathsf{V}(S_4, S_9)$.
On the other hand, it is easy to see that both $L_2$ and $M_2$ can be embedded into $S_9$.
Also, $S_9$ satisfies an equational basis of $\mathsf{V}(L_2, M_2)$ that can be found in \cite{sr}.
It follows that $\mathsf{V}(S_9)=\mathsf{V}(L_2, M_2)$
and so $\mathsf{V}(S_{(4, 47)})=\mathsf{V}(S_4, L_2, M_2)$.
Since $M_2$ can be embedded into $S_4$, we therefore have that
\[
\mathsf{V}(S_{(4, 47)})=\mathsf{V}(S_4, L_2).
\]
\end{remark}

\begin{cor}
The ai-semiring $S_{(4, 21)}$ is finitely based.
\end{cor}
\begin{proof}
It is easy to see that $S_{(4, 21)}$ and $S_{(4, 47)}$ have dual multiplications.
By Proposition \ref{pro24061701} we immediately deduce that $S_{(4, 21)}$ is finitely based.
\end{proof}

\begin{pro}\label{pro24071020}
$\mathsf{V}(S_{(4, 48)})$ is the ai-semiring variety defined by the identities
\begin{align}
x^2y& \approx xy; \label{f24062101}\\
xyz& \approx yxz; \label{f24062102}\\
x^2&\approx x^2+x;\label{f24062105}\\
x+y^2& \approx x+y^2+x^2y^2;    \label{f24062103}\\
x+yz&\approx x+yz+yx; \label{f24062104}\\
x+xyz& \approx x+xyz+xy;  \label{f24062106}\\
z+xyz& \approx z+xyz+yz.  \label{f24062107}
\end{align}
\end{pro}
\begin{proof}
It is easy to verify that $S_{(4, 48)}$ satisfies the identities
(\ref{f24062101})--(\ref{f24062107}).
In the remainder we only need to show that every ai-semiring identity of $S_{(4, 48)}$
is derivable from (\ref{f24062101})--(\ref{f24062107}) and the identities defining $\mathbf{AI}$.
Let $\bu \approx \bu+\bq$ be such a nontrivial identity, where $\bu=\bu_1+\bu_2+\cdots+\bu_n$ and $\bu_i, \bq\in X^+$, $1 \leq i \leq n$.
Since $D_2$ is isomorphic to $\{1, 4\}$, we have that $D_2$ satisfies $\bu \approx \bu+\bq$ and so
$c(\bu_i)\subseteq c(\bq)$ for some $\bu_i \in \bu$.
On the other hand, it is easy to see that $S_4$ is isomorphic to $\{1, 2, 3\}$
and so $S_4$ satisfies $\bu \approx \bu+\bq$.
By Lemma \ref{lem24061501} we have that $c(\bq)\subseteq  c(\bu)$,
$\ell(\bu_j)\geq 2$ for some $\bu_j \in \bu$, and $\bu$ satisfies the property {\rm (T)}
implies that $\bu+\bq$ satisfies the property {\rm (T)}.

\textbf{Case 1.} $\bu$ satisfies the property {\rm (T)}.
Then $\bu+\bq$ also satisfies the property {\rm (T)}.
This implies that $t(\bq)\notin c(p(\bq))$ and $t(\bq)=t(\bu_i)$.
By the identities (\ref{f24062101}) and (\ref{f24062102}) we derive
\begin{equation}\label{f240620001}
\bq \approx x_1x_2\cdots x_{k}t(\bq),
\end{equation}
where $\{x_1, x_2, \ldots, x_k\}=c(p(\bq))$, $k \geq 0$.
We may suppose that
$\{x_1, \ldots, x_r\}$ and $\{y_1, \ldots, y_s\}$ are two disjoint subsets of $c(\bu)$ such that
\[
c(\bu)=\{x_1, \ldots, x_r\} \bigcup \{y_1, \ldots, y_s\},
\]
and
\[
t(\bu)=\{y_1, \ldots, y_s\},
\]
where $y_1=t(\bq)$.
Now we have
\begin{align*}
\bu
&\approx \bu+x_1x_2\cdots x_kx_{k+1}x_{k+2}\cdots x_rt(\bq)&&(\text{by}~(\ref{f24062101}), (\ref{f24062102}), (\ref{f24062104}))\\
&\approx \bu+\bu_i+\bp p(\bq) \bu_i&&  (\text{by}~(\ref{f24062101}), (\ref{f24062102}))                   \\
&\approx \bu+\bu_i+\bp p(\bq) \bu_i+p(\bq)\bu_i &&  (\text{by}~  (\ref{f24062107})) \\
&\approx \bu+p(\bq)\bu_i\\
&\approx \bu+\bq. &&  (\text{by}~  (\ref{f24062101}), (\ref{f24062102}))
\end{align*}

\textbf{Case 2.} $\bu$ does not satisfy the property {\rm (T)}.
Consider the following two subcases.

\textbf{Subcase 2.1.} $m(t(\bu_{i_1}), \bu_{i_2}) \geq 2$ for some $\bu_{i_1}, \bu_{i_2} \in \bu$.
Then $t(\bu_{i_1}) \in c(p(\bu_{i_2}))$ and so we have
\begin{align*}
\bu
&\approx \bu+\bu_{i_1}+\bu_{i_2} \\
&\approx \bu+\bu_{i_1}+p(\bu_{i_2})t(\bu_{i_2})\\
&\approx \bu+\bu_{i_1}+p(\bu_{i_2})t(\bu_{i_2})+p(\bu_{i_2})\bu_{i_1}. &&(\text{by}~(\ref{f24062104}))
\end{align*}
This implies the identity
\begin{equation}\label{f24071501}
\bu\approx \bu+p(\bu_{i_2})\bu_{i_1}.
\end{equation}
Furthermore, we can deduce
\begin{align*}
\bu
&\approx \bu+p(\bu_{i_2})\bu_{i_1}\\
&\approx \bu+(p(\bu_{i_2})\bu_{i_1})^2 &&(\text{by}~(\ref{f24062101}), (\ref{f24062102}))\\
&\approx \bu+\bp^2 &&(\text{by}~(\ref{f24062102}), (\ref{f24062103}))\\
&\approx \bu+\bu_i+\bu_i\bq^2\bp^2&&(\text{by}~(\ref{f24062101}), (\ref{f24062102}))\\
&\approx \bu+\bu_i+\bu_i\bq^2\bp^2+\bu_i\bq^2&&(\text{by}~  (\ref{f24062106})) \\
&\approx \bu+\bu_i+\bu_i\bq^2\bp^2+\bq^2  &&(\text{by}~(\ref{f24062101}), (\ref{f24062102})) \\
&\approx \bu+\bu_i+\bu_i\bq^2\bp^2+\bq^2+\bq, &&(\text{by}~  (\ref{f24062105}))
\end{align*}
where $c(\bp)=c(\bu)$. So we obtain the identity $\bu\approx \bu+\bq$.

\textbf{Subcase 2.2.} $m(t(\bu_{i_1}), \bu_{i_2}) = 1$ and $t(\bu_{i_1})\neq t(\bu_{i_2})$ for some $\bu_{i_1}, \bu_{i_2} \in \bu$.
This is similar to the preceding case.
\end{proof}

\begin{remark}
It is easily verified that $S_{(4, 48)}$ is isomorphic to a subdirect product of
$S_4$ and $S_{15}$. So $\mathsf{V}(S_{(4, 48)})=\mathsf{V}(S_4, S_{15})$.
On the other hand, we have that both $M_2$ and $D_2$ can be embedded into $S_{15}$
and that $S_{15}$ satisfies an equational basis of $\mathsf{V}(M_2, D_2)$,
which can be found in \cite{sr}.
This implies that $\mathsf{V}(S_{15})=\mathsf{V}(M_2, D_2)$.
Since $M_2$ can be embedded into $S_4$, we deduce that
\[
\mathsf{V}(S_{(4, 48)})=\mathsf{V}(S_4, D_2).
\]
\end{remark}

Notice that $S_{(4, 46)}$ and $S_{(4, 48)}$ have dual multiplications.
By Proposition \ref{pro24071020} we immediately deduce
\begin{cor}
$S_{(4, 46)}$ is finitely based.
\end{cor}

\section{Equational basis of $S_{(4, 12)}$}
In this section we provide a finite equational basis of $S_{(4, 12)}$ and show that $S_{(4, 12)}$ is finitely based.
One can easily find that $S_{(4, 12)}$ is isomorphic to a subdirect product of $S_4$ and $S_6$.
So $\mathsf{V}(S_{(4, 12)})=\mathsf{V}(S_4, S_6)$.
Notice that $S_6$ and $S_4$ have dual multiplications.
By Lemma \ref{lem24061501} we can deduce the following result,
which provides a solution of the equational problem for $S_6$.
\begin{lem}\label{lem24063001}
Let $\bu\approx \bu+\bq$ be a nontrivial ai-semiring identity such that
$\bu=\bu_1+\bu_2+\cdots+\bu_n$, where $\bu_i, \bq \in X^+$, $1\leq i \leq n$.
Then $\bu\approx \bu+\bq$ is satisfied by $S_6$ if and only if $\bu$ and $\bq$ satisfy the following conditions:
\begin{itemize}
\item[$(1)$] $c(\bq)\subseteq  c(\bu)$;

\item[$(2)$] $\ell(\bu_i)\geq 2$ for some $\bu_i \in \bu$;

\item[$(3)$] If $\bu$ satisfies the property {\rm (H)}:
\[
(\forall \bu_i, \bu_j \in \bu) ~ m(h(\bu_i), \bu_j)\leq 1; m(h(\bu_i), \bu_j)= 1\Rightarrow h(\bu_i)=h(\bu_j),
\]
then $\bu+\bq$ also satisfies the property {\rm (H)}, that is,
\[
(\forall \bu_i \in \bu) ~ m(h(\bu_i), \bq)\leq 1; m(h(\bu_i), \bq)= 1\Rightarrow h(\bu_i)=h(\bq);
\]
\[
m(h(\bq), \bq)= 1; m(h(\bq), \bu_i)\leq 1; m(h(\bq), \bu_i)= 1\Rightarrow h(\bq)=h(\bu_i).
\]
\end{itemize}
\end{lem}

\begin{pro}
$\mathsf{V}(S_{(4, 12)})$ is the ai-semiring variety defined by the identities
\begin{align}
x^2 & \approx x^4; \label{f24070371}\\
x^2y^2 & \approx (xy)^2; \label{f24070370}\\
x^2y^2& \approx y^2x^2;    \label{f24071533}                \\
x^2& \approx x^2+x; \label{f24062923}\\
x^2y^2& \approx x^2y^2+x^2; \label{f24070372}\\
x+yx  & \approx y^2x;\label{f24062931}\\
x+xy  & \approx xy^2;\label{f24062932}\\
x+y^2& \approx x+y^2+y^2x^2; \label{f24062925}\\
xy^2+z& \approx xy^2+zy;              \label{f24070210}    \\
x^2y+z& \approx x^2y+xz; \\
x_1^2x_2+x_3x_4^2 & \approx x_1^2x_2^2x_3^2x_4^2; &\label{f24062927}\\
x_1x_2+y_1y_2& \approx x_1x_2+y_1y_2+x_1y_2; \label{f24062924}\\
x_1x_2+x_3x_2x_4& \approx x_1x_2+x_3x_2x_4+x_1; \label{f24062926}
\end{align}
\begin{align}
x_1x_2+x_3x_1x_4& \approx x_1x_2+x_3x_1x_4+x_2;\\
x_1x_2+y_1x_2y_2       & \approx x_1x_2+y_1x_2y_2+x_1x_2y_2^2;\label{f24070320}\\
x_1x_2+y_1x_1y_2       & \approx x_1x_2+y_1x_1y_2+y_1^2x_1x_2; \label{f24070321}\\
x_1x_2+x_3x_2x_4+x_5& \approx x_1x_2+x_3x_2x_4+x_5x_2; \label{f24070201}\\
x_1x_2+x_3x_1x_4+x_5& \approx x_1x_2+x_3x_1x_4+x_1x_5; \label{f240702001}\\
x_1x_2+y_1x_1y_2x_1y_3 & \approx x_1x_2+y_1x_1y_2x_1y_3+x_1^2x_2; \label{f24070301}\\
x_1x_2+y_1x_2y_2x_2y_3 & \approx x_1x_2+y_1x_2y_2x_2y_3+x_1x_2^2, \label{f24070302}
\end{align}
where $x_2$ and $x_3$ may be empty in $(\ref{f24062927})$, $x_1$ and $y_1$ may be empty in $(\ref{f24070320})$,
$x_2$ and $y_2$ may be empty in $(\ref{f24070321})$,
$x_2$, $y_1$, $y_2$ and $y_3$ may be empty in $(\ref{f24070301})$,
and $x_1$, $y_1$, $y_2$ and $y_3$ may be empty in $(\ref{f24070302})$.
\end{pro}
\begin{proof}
It is easy to verify that $S_{(4, 12)}$ satisfies the identities
(\ref{f24070371})--(\ref{f24070302}).
In the remainder it is enough to show that every ai-semiring identity of $S_{(4, 12)}$
is derivable from (\ref{f24070371})--(\ref{f24070302}) and the identities defining $\mathbf{AI}$.
Let $\bu \approx \bu+\bq$ be such an identity, where $\bu=\bu_1+\bu_2+\cdots+\bu_n$ and $\bu_i, \bq\in X^+$, $1 \leq i \leq n$.
Then both $S_4$ and $S_6$ satisfy $\bu \approx \bu+\bq$.
By Lemmas \ref{lem24061501} and \ref{lem24063001} we need to consider the following cases.

\textbf{Case 1.} $\bu$ satisfies the properties {\rm (H)} and {\rm (T)}.
Then $\bu+\bq$ satisfies the properties {\rm (H)} and {\rm (T)}.
Since $c(\bq) \subseteq c(\bu)$, it follows that
there exist $\bu_i,  \bu_j \in \bu$ such that $h(\bu_i)=h(\bq)$ and $t(\bu_j)=t(\bq)$.
This implies that $\bu_i=h(\bq)s(\bu_i)$ and $\bu_j=p(\bu_j)t(\bq)$.

\textbf{Subcase 1.1.} $\ell(\bq)=1$. Then $h(\bq)=t(\bq)=\bq$.
If $\ell(\bu_i)=1$ or $\ell(\bu_j)=1$,
then it is easy to see that $\bu\approx \bu+\bq$ is trivial.
Now assume that $\ell(\bu_i)\geq 2$ and $\ell(\bu_j)\geq 2$.
Then both $s(\bu_i)$ and $p(\bu_j)$ are nonempty.
Furthermore, we have
\begin{align*}
\bu
&\approx \bu+\bu_i+\bu_j&&\\
&\approx \bu+h(\bq)s(\bu_i)+p(\bu_j)t(\bq)&&       \\
&\approx \bu+\bq s(\bu_i)+p(\bu_j)\bq          \\
&\approx \bu+\bq s(\bu_i)+p(\bu_j)\bq+\bq^2 &&   (\text{by}~(\ref{f24062924}))\\
&\approx \bu+\bq s(\bu_i)+p(\bu_j)\bq+\bq^2+\bq.  &&  (\text{by}~(\ref{f24062923}))
\end{align*}
This derives the identity $\bu \approx \bu+\bq$.

\textbf{Subcase 1.2.} $\ell(\bq)\geq 2$.
Let $\bq=h(\bq) x_1 \cdots x_m t(\bq)$,
where $m\geq 0$, $h(\bq) \neq t(\bq)$, $x_i\neq h(\bq)$ and $x_i\neq t(\bq)$ for all $1 \leq i \leq n$.
Suppose by way of contradiction that $\ell(\bu_i)=1$. Then $h(\bu_i)=t(\bu_i)=\bu_i$.
Since $h(\bu_i)=h(\bq)$, it follows that $t(\bu_i)=h(\bq)$.
This implies that $t(\bu_i)\in c(\bq)$ and so $t(\bu_i)=t(\bq)$.
Thus $h(\bq)=t(\bq)$, a contradiction. So $\ell(\bu_i)\geq 2$.
Similarly, one can show that $\ell(\bu_j)\geq 2$.
Let $1\leq i \leq m$. Since $c(\bq) \subseteq c(\bu)$, there exists $\bu_{r_i} \in \bu$ such that
$x_i \in p(\bu_{r_i})\cap s(\bu_{r_i})$ and so $\ell(\bu_{r_i})\geq 3$.
This implies that $\bu_{r_i}=\bu_{r_i}'x_i\bu_{r_i}''$
for some nonempty words $\bu_{r_i}'$ and $\bu_{r_i}''$.

Next, we shall show by induction on $k$ that the identities
(\ref{f24070371})--(\ref{f24070302}) can derive
$\bu \approx \bu+h(\bq)x_1x_2\cdots x_k t(\bq)$ for all $0\leq k\leq m$.
Indeed, if $k=0$, then
\begin{align*}
\bu
&\approx \bu+\bu_i+\bu_j&&\\
&\approx \bu+h(\bq)s(\bu_i)+p(\bu_j)t(\bq)&&       \\
&\approx \bu+h(\bq)s(\bu_i)+p(\bu_j)t(\bq)+h(\bq)t(\bq). &&   (\text{by}~(\ref{f24062924}))
\end{align*}
This implies the identity $\bu\approx \bu+h(\bq)t(\bq)$.
If $k=1$, then

\begin{align*}
\bu
&\approx \bu+\bu_i+\bu_{r_1}   &&         \\
&\approx \bu+h(\bu_i)s(\bu_i)+\bu_{r_1}'x_1\bu_{r_1}''                    &&       \\
&\approx \bu+h(\bq)s(\bu_i)+\bu_{r_1}'x_1\bu_{r_1}''                    &&       \\
&\approx \bu+h(\bq)s(\bu_i)+\bu_{r_1}'x_1\bu_{r_1}''+h(\bq)x_1\bu_{r_1}''  && (\text{by}~(\ref{f24062924})) \\
&\approx \bu+h(\bq)s(\bu_i)+\bu_{r_1}'x_1\bu_{r_1}''+h(\bq)x_1\bu_{r_1}''+\bu_j\\
&\approx \bu+h(\bq)s(\bu_i)+\bu_{r_1}'x_1\bu_{r_1}''+h(\bq)x_1\bu_{r_1}''+p(\bu_j)t(\bu_j)\\
&\approx \bu+h(\bq)s(\bu_i)+\bu_{r_1}'x_1\bu_{r_1}''+h(\bq)x_1\bu_{r_1}''+p(\bu_j)t(\bq)\\
&\approx \bu+h(\bq)s(\bu_i)+\bu_{r_1}'x_1\bu_{r_1}''+h(\bq)x_1\bu_{r_1}''+p(\bu_j)t(\bq)+h(\bq)x_1t(\bq). && (\text{by}~(\ref{f24062924}))
\end{align*}
This implies the identity $\bu\approx \bu+h(\bq)x_1t(\bq)$.
Let $2\leq k\leq m$. Suppose that the identities
(\ref{f24070371})--(\ref{f24070302}) can derive
$\bu \approx \bu+h(\bq)x_1x_2\cdots x_{k-1} t(\bq)$.
Then
\begin{align*}
\bu
&\approx \bu+h(\bq)x_1x_2\cdots x_{k-1}t(\bq)+\bu_{r_k}   &&         \\
&\approx \bu+h(\bq)x_1x_2\cdots x_{k-1}t(\bq)+\bu_{r_k}'x_k\bu_{r_k}''   &&        \\
&\approx \bu+h(\bq)x_1x_2\cdots x_{k-1}t(\bq)+\bu_{r_k}'x_k\bu_{r_k}''+h(\bq)x_1x_2\cdots x_{k-1}x_k\bu_{r_k}''. &&(\text{by}~(\ref{f24062924}))
\end{align*}
This implies the identity
\begin{equation*}
\bu \approx \bu+h(\bq)x_1x_2\cdots x_{k-1}x_k\bu_{r_k}''.
\end{equation*}
Furthermore, we have
\begin{align*}
\bu
&\approx\bu+h(\bq)x_1x_2\cdots x_{k-1}x_k\bu_{r_k}''  \\
&\approx \bu+h(\bq)x_1x_2\cdots x_{k-1}x_k\bu_{r_k}''+\bu_j \\
&\approx \bu+h(\bq)x_1x_2\cdots x_{k-1}x_k\bu_{r_k}''+p(\bu_j)t(\bq)\\
&\approx \bu+h(\bq)x_1x_2\cdots x_{k-1}x_k\bu_{r_k}''+p(\bu_j)t(\bq)+h(\bq)x_1x_2\cdots x_{k-1}x_kt(\bq).&& (\text{by}~(\ref{f24062924}))
\end{align*}
This proves the identity
\[
\bu \approx \bu+h(\bq)x_1x_2\cdots x_{k-1}x_kt(\bq).
\]
Take $k=m$. We obtain the identity
\[
\bu\approx \bu+h(\bq)x_1x_2\cdots x_{k-1}x_mt(\bq)\approx \bu+\bq.
\]

\textbf{Case 2.} $\bu$ satisfies the property {\rm (H)}, but does not satisfy the property {\rm (T)}.
Then $\bu+\bq$ satisfies the property {\rm (H)}.
Let $\bq=h(\bq) x_1 \cdots x_m$, where $m \geq 0$.
Then $h(\bq) \neq x_i$ for all $1 \leq i \leq m$.
Since $c(\bq)\subseteq c(\bu)$,
it follows that there exists $\bu_i\in \bu$ such that
$h(\bu_i)=h(\bq)$ and so $\bu_i=h(\bq)s(\bu_i)$, where $s(\bu_i)$ may be empty.
For any $1 \leq i \leq m$,
there exists $\bu_{r_i} \in \bu$ such that
$x_i \in s(\bu_{r_i})$ and so $\ell(\bu_{r_i})\geq 2$.
This implies that $\bu_{r_i}=\bu_{r_i}'x_i\bu_{r_i}''$
for some nonempty word $\bu_{r_i}'$ and some word $\bu_{r_i}''$, which may be empty.
Since $\bu$ does not satisfy the property {\rm (T)},
we need to consider the following two subcases.

\textbf{Subcase 2.1.} $m(t(\bu_{i_1}), \bu_{i_2})\geq 2$ for some $\bu_{i_1}, \bu_{i_2} \in \bu$.
Suppose that $\ell(\bu_{i_1})=1$. Then $h(\bu_{i_1})=t(\bu_{i_1})$ and so $m(h(\bu_{i_1}), \bu_{i_2})\geq 2$.
This contradicts the fact that $\bu$ satisfies {\rm (H)}.
Thus $\ell(\bu_{i_1})\geq 2$ and so $p(\bu_{i_1})$ is nonempty.
Moreover, we can deduce that $h(\bu_{i_2})\neq t(\bu_{i_1})$ and so $\bu_{i_2}=\bu_{i_2}'t(\bu_{i_1})\bu_{i_2}''$
for some nonempty words $\bu_{i_2}'$ and $\bu_{i_2}''$.
Now we have

\begin{align*}
\bu
&\approx \bu+\bu_{i_1}+\bu_{i_2}+\bu_i  &&         \\
&\approx \bu+p(\bu_{i_1})t(\bu_{i_1})+\bu_{i_2}'t(\bu_{i_1})\bu_{i_2}''+h(\bq)s(\bu_i)&&         \\
&\approx \bu+p(\bu_{i_1})t(\bu_{i_1})+\bu_{i_2}'t(\bu_{i_1})\bu_{i_2}''+h(\bq)s(\bu_i)t(\bu_{i_1}). &&(\text{by}~(\ref{f24070201}))
\end{align*}
So we may assume that $s(\bu_i)$ is nonempty.
Similarly, for any $1\leq i \leq m$, we have
\begin{align*}
\bu
&\approx \bu+\bu_{i_1}+\bu_{i_2}+\bu_{r_i}  &&         \\
&\approx \bu+p(\bu_{i_1})t(\bu_{i_1})+\bu_{i_2}'t(\bu_{i_1})\bu_{i_2}''+\bu_{r_i}'x_i\bu_{r_i}''&&         \\
&\approx \bu+p(\bu_{i_1})t(\bu_{i_1})+\bu_{i_2}'t(\bu_{i_1})\bu_{i_2}''+\bu_{r_i}'x_i\bu_{r_i}''t(\bu_{i_1}). &&(\text{by}~(\ref{f24070201}))
\end{align*}
From this observation we may assume that $\bu_{r_i}''$ is nonempty for all $1\leq i \leq m$.
One can prove by induction on $k$ that
the identities (\ref{f24070371})--(\ref{f24070302}) can derive
$\bu \approx \bu+h(\bq)x_1x_2\cdots x_{k}\bu_{r_k}''$ for all $1\leq k\leq m$.
This process is similar to Subcase 1.2.
Put $k=m$. We obtain the identity
\begin{equation*}
\bu \approx \bu+h(\bq)x_1x_2\cdots x_{m}\bu_{r_m}''.
\end{equation*}
Now we have
\begin{align*}
\bu
&\approx \bu+h(\bq)x_1x_2\cdots x_{m}\bu_{r_m}''+\bu_{i_1}   &&         \\
&\approx \bu+h(\bq)x_1x_2\cdots x_{m}\bu_{r_m}''+p(\bu_{i_1})t(\bu_{i_1})                    &&       \\
&\approx \bu+h(\bq)x_1x_2\cdots x_{m}\bu_{r_m}''+p(\bu_{i_1})t(\bu_{i_1})+ h(\bq)x_1x_2\cdots x_{m}t(\bu_{i_1}).   &&  (\text{by}~(\ref{f24062924}))
\end{align*}
This implies the identity
\[
\bu \approx \bu+h(\bq)x_1x_2\cdots x_{m}t(\bu_{i_1}).
\]
Furthermore, we have
\begin{align*}
\bu
&\approx \bu+h(\bq)x_1x_2\cdots x_{m}t(\bu_{i_1})  &&  \\
&\approx \bu+h(\bq)x_1x_2\cdots x_{m}t(\bu_{i_1})+\bu_{i_2}\\
&\approx \bu+h(\bq)x_1x_2\cdots x_{m}t(\bu_{i_1})+\bu_{i_2}'t(\bu_{i_1})\bu_{i_2}''\\
&\approx \bu+h(\bq)x_1x_2\cdots x_{m}t(\bu_{i_1})+
\bu_{i_2}'t(\bu_{i_1})\bu_{i_2}''+h(\bq)x_1x_2\cdots x_{m}. &&(\text{by}~(\ref{f24062926}))
\end{align*}
This derives
\[
\bu\approx \bu+h(\bq)x_1x_2\cdots x_{m}\approx \bu+\bq.
\]

\textbf{Subcase 2.2.}
$m(t(\bu_{i_1}), \bu_{i_2}) = 1$ and $t(\bu_{i_1})\neq t(\bu_{i_2})$ for some $\bu_{i_1}, \bu_{i_2} \in \bu$.
Then $t(\bu_{i_1})\in c(p(\bu_{i_2}))$.
If $\ell(\bu_{i_1})\geq 2$, then $h(\bu_{i_1})\neq t(\bu_{i_1})$.
This implies that $h(\bu_{i_2})\neq t(\bu_{i_1})$
and so $\bu_{i_2}=\bu_{i_2}'t(\bu_{i_1})\bu_{i_2}''$
for some nonempty words $\bu_{i_2}'$ and $\bu_{i_2}''$.
The remaining process is similar to Subcase 2.1.
Now assume that $\ell(\bu_{i_1})=1$.
Then $t(\bu_{i_1})=h(\bu_{i_1})=\bu_{i_1}$ and so $h(\bu_{i_2})=\bu_{i_1}$.
This implies that $\bu_{i_2}=\bu_{i_1}s(\bu_{i_2})$, where $s(\bu_{i_2})$ is nonempty.
By the identity (\ref{f24062932}) we deduce
\begin{equation*}\label{f24070211}
\bu\approx\bu+\bu_{i_1}+\bu_{i_2}\approx \bu+\bu_{i_1}+\bu_{i_1}s(\bu_{i_2})\approx \bu+\bu_{i_1}s(\bu_{i_2})s(\bu_{i_2})
\end{equation*}
and so we obtain
\begin{equation}\label{f24071522}
\bu\approx\bu+\bu_{i_1}s(\bu_{i_2})s(\bu_{i_2}).
\end{equation}
Furthermore, we have

\begin{align*}
\bu
&\approx \bu+\bu_{i_1}s(\bu_{i_2})s(\bu_{i_2})&&(\text{by}~(\ref{f24071522}))\\
&\approx \bu+\bu_{i_1}s(\bu_{i_2})s(\bu_{i_2})+\bu_i &&         \\
&\approx \bu+\bu_{i_1}s(\bu_{i_2})s(\bu_{i_2})+h(\bq)s(\bu_i) &&         \\
&\approx \bu+\bu_{i_1}s(\bu_{i_2})s(\bu_{i_2})+h(\bq)s(\bu_i)s(\bu_{i_2}). &&(\text{by}~(\ref{f24070210}))
\end{align*}
So we may assume that $s(\bu_i)$ is nonempty.
Similarly, for any $1\leq i \leq m$, we have
\begin{align*}
\bu
&\approx \bu+\bu_{i_1}s(\bu_{i_2})s(\bu_{i_2})&&(\text{by}~(\ref{f24071522}))\\
&\approx \bu+\bu_{i_1}s(\bu_{i_2})s(\bu_{i_2})+\bu_{r_i} &&         \\
&\approx \bu+\bu_{i_1}s(\bu_{i_2})s(\bu_{i_2})+\bu_{r_i}'x_i\bu_{r_i}'' &&         \\
&\approx \bu+\bu_{i_1}s(\bu_{i_2})s(\bu_{i_2})+\bu_{r_i}'x_i\bu_{r_i}''s(\bu_{i_2}). &&(\text{by}~(\ref{f24070210}))
\end{align*}
From this result we may assume that $\bu_{r_i}''$ is nonempty for all $1\leq i \leq m$.
One can prove by induction on $k$ that
the identities (\ref{f24070371})--(\ref{f24070302}) can derive
$\bu \approx \bu+h(\bq)x_1x_2\cdots x_{k}\bu_{r_k}''$ for all $1\leq k\leq m$.
This process is similar to Subcase 1.2.
Put $k=m$. We obtain
\begin{equation*}
\bu \approx \bu+h(\bq)x_1x_2\cdots x_{m}\bu_{r_m}''.
\end{equation*}
Furthermore, we have
\begin{align*}
\bu
&\approx \bu+h(\bq)x_1x_2\cdots x_{m}\bu_{r_m}''+\bu_{i_2}   &&         \\
&\approx \bu+h(\bq)x_1x_2\cdots x_{m}\bu_{r_m}''+\bu_{i_1}s(\bu_{i_2})   &&         \\
&\approx \bu+h(\bq)x_1x_2\cdots x_{m}\bu_{r_m}''+\bu_{i_1}s(\bu_{i_2})+h(\bq)x_1x_2\cdots x_{m}s(\bu_{i_2})
&& (\text{by}~(\ref{f24062924}))
\end{align*}
This proves the identity
\begin{equation}\label{f24070350}
\bu\approx \bu+h(\bq)x_1x_2\cdots x_{m}s(\bu_{i_2}).
\end{equation}
On the other hand, we have
\[
\bu\approx \bu+\bu_{i_2}+\bu_{i_1}\approx\bu+\bu_{i_1}s(\bu_{i_2})+\bu_{i_1}\stackrel{(\ref{f24062932})}\approx\bu+\bu_{i_1}s(\bu_{i_2}))s(\bu_{i_2})
\]
and so the identity
\begin{equation}\label{f24070351}
\bu\approx \bu+\bu_{i_1}s(\bu_{i_2}))s(\bu_{i_2})
\end{equation}
is derived. Now we have
\begin{align*}
\bu
&\approx \bu+h(\bq)x_1x_2\cdots x_{m}s(\bu_{i_2})+\bu_{i_1}s(\bu_{i_2})s(\bu_{i_2})&&(\text{by}~(\ref{f24070350}), (\ref{f24070351}))\\
&\approx \bu+h(\bq)x_1x_2\cdots x_{m}s(\bu_{i_2})+\bu_{i_1}s(\bu_{i_2})s(\bu_{i_2})+h(\bq)x_1x_2\cdots x_{m}.&&(\text{by}~(\ref{f24062926}))
\end{align*}
This derives the identity
\[
\bu\approx \bu+h(\bq)x_1x_2\cdots x_{m}.
\]
So we obtain the identity
\[
\bu\approx \bu+\bq.
\]

\textbf{Case 3.} $\bu$ satisfies the property {\rm (T)}, but does not satisfy the property {\rm (H)}.
This is similar to Case 2.

\textbf{Case 4.} $\bu$ does not satisfy the property {\rm (H)} or the property {\rm (T)}.

\textbf{Subcase 4.1.} $m(h(\bu_{i_1}), \bu_{i_2})\geq 2$ and $m(t(\bu_{j_1}), \bu_{j_2})\geq 2$
for some $\bu_{i_1}, \bu_{i_2}, \bu_{j_1}$, $\bu_{j_2} \in \bu$.
Then $\bu_{i_1}=h(\bu_{i_1})s(\bu_{i_1})$, $\bu_{i_2}=\bp_1h(\bu_{i_1})\bp_2h(\bu_{i_1})\bp_3$,
$\bu_{j_1}=p(\bu_{j_1})t(\bu_{j_1})$ and $\bu_{j_2}=\bq_1t(\bu_{j_1})\bq_2t(\bu_{j_1})\bq_3$,
where $s(\bu_{i_1})$, $\bp_1$, $\bp_2$, $\bp_3$, $p(\bu_{j_1})$, $\bq_1$, $\bq_2$, $\bq_3$
may be empty.
Furthermore, we have
\begin{align*}
\bu
&\approx \bu+\bu_{i_1}+\bu_{i_2}                \\
&\approx \bu+h(\bu_{i_1})s(\bu_{i_1})+\bp_1h(\bu_{i_1})\bp_2h(\bu_{i_1})\bp_3 \\
&\approx \bu+h(\bu_{i_1})s(\bu_{i_1})+\bp_1h(\bu_{i_1})\bp_2h(\bu_{i_1})\bp_3+h(\bu_{i_1})h(\bu_{i_1})s(\bu_{i_1}). &&(\text{by}~(\ref{f24070301}))
\end{align*}
This derives the identity
\begin{equation}\label{f24070305}
\bu \approx \bu+h(\bu_{i_1})h(\bu_{i_1})s(\bu_{i_1}).
\end{equation}
Similarly, we have
\begin{align*}
\bu
&\approx \bu+\bu_{j_1}+\bu_{j_2}                \\
&\approx \bu+p(\bu_{j_1})t(\bu_{j_1})+\bq_1t(\bu_{j_1})\bq_2t(\bu_{j_1})\bq_3 \\
&\approx \bu+p(\bu_{j_1})t(\bu_{j_1})+\bq_1t(\bu_{j_1})\bq_2t(\bu_{j_1})\bq_3+p(\bu_{j_1})t(\bu_{j_1})t(\bu_{j_1}). &&(\text{by}~(\ref{f24070302}))
\end{align*}
So we obtain the identity
\begin{equation}\label{f24070306}
\bu \approx \bu+p(\bu_{j_1})t(\bu_{j_1})t(\bu_{j_1}).
\end{equation}
Now we have
\begin{align*}
\bu
&\approx \bu+h(\bu_{i_1})h(\bu_{i_1})s(\bu_{i_1})+p(\bu_{j_1})t(\bu_{j_1})t(\bu_{j_1})
&&(\text{by}~(\ref{f24070305}), (\ref{f24070306}))              \\
&\approx \bu+h(\bu_{i_1})h(\bu_{i_1})s(\bu_{i_1})s(\bu_{i_1})p(\bu_{j_1})p(\bu_{j_1})t(\bu_{j_1})t(\bu_{j_1}) &&(\text{by}~(\ref{f24062927}))\\
&\approx \bu+x_1^2x_2^2\cdots x_m^2 &&(\text{by}~(\ref{f24070370}),(\ref{f24062925}))\\
&\approx \bu+\bq^2\bp^2 &&(\text{by}~(\ref{f24070371}),(\ref{f24070370}), (\ref{f24071533}))\\
&\approx \bu+\bq^2\bp^2+\bq^2&&(\text{by}~(\ref{f24070372})) \\
&\approx \bu+\bq^2\bp^2+\bq^2+\bq, &&(\text{by}~(\ref{f24062923}))
\end{align*}
where $c(\bu)=\{x_1, x_2, \ldots, x_m\}$.
This derives the identity $\bu\approx \bu+\bq$.

\textbf{Subcase 4.2.}
$m(h(\bu_{i_1}), \bu_{i_2})\geq 2$, $m(t(\bu_{j_1}), \bu_{j_2}) = 1$ and $t(\bu_{j_1})\neq t(\bu_{j_2})$
for some $\bu_{i_1}, \bu_{i_2}, \bu_{j_1}, \bu_{j_2} \in \bu$.
Then
$\bu_{j_1}=p(\bu_{j_1})t(\bu_{j_1})$ and $\bu_{j_2}=\bq_1t(\bu_{j_1})\bq_2$
where $\bq_2$ is nonempty, $p(\bu_{j_1})$ and $\bq_1$ may be empty.
We have
\begin{align*}
\bu
&\approx \bu+\bu_{j_1}+\bu_{j_2}                \\
&\approx \bu+p(\bu_{j_1})t(\bu_{j_1})+\bq_1t(\bu_{j_1})\bq_2 \\
&\approx \bu+p(\bu_{j_1})t(\bu_{j_1})+\bq_1t(\bu_{j_1})\bq_2+p(\bu_{j_1})t(\bu_{j_1})\bq_2^2. &&(\text{by}~(\ref{f24070320}))
\end{align*}
This derives the identity
\[
\bu \approx \bu+p(\bu_{j_1})t(\bu_{j_1})\bq_2^2,
\]
where $m(t(\bu_{j_2}), p(\bu_{j_1})t(\bu_{j_1})\bq_2^2)\geq 2$.
The remaining is similar to Subcase 4.1.

\textbf{Subcase 4.3.}
$m(h(\bu_{i_1}), \bu_{i_2}) = 1$, $h(\bu_{i_1})\neq h(\bu_{i_2})$ and $m(t(\bu_{j_1}), \bu_{j_2})\geq 2$
for some $\bu_{i_1}, \bu_{i_2}, \bu_{j_1}, \bu_{j_2} \in \bu$.
The process is similar to Subcases 4.1 and 4.2.

\textbf{Subcase 4.4.}
$m(h(\bu_{i_1}), \bu_{i_2}) = 1$, $h(\bu_{i_1})\neq h(\bu_{i_2})$, $m(t(\bu_{j_1}), \bu_{j_2}) = 1$ and $t(\bu_{j_1})\neq t(\bu_{j_2})$
for some $\bu_{i_1}, \bu_{i_2}, \bu_{j_1}, \bu_{j_2} \in \bu$.
The process is similar to Subcases 4.1 and 4.2.
\end{proof}

\section{Conclusion}
We have answered the finite basis problem for $4$-element ai-semirings whose additive reducts are semilattices of height $1$.
We shall proceed to study the corresponding problem for remaining $4$-element ai-semirings.
From the whole paper one can further conclude that flat semirings are very important in the theory of ai-semiring varieties.
Moreover, it is of interest to study the relationship among the finite basis problem for $S_{ne}$, $S_{ie}$ and $S$.

\qquad

\noindent
\textbf{Acknowledgements}
The authors thank their team members Zidong Gao, Zexi Liu, Qizheng Sun, Mengyu Yuan and Mengya Yue for discussions contributed to this paper.
The authors are particularly grateful to Professor Mikhail Volkov for his helpful suggestions.
The authors also thank the anonymous referee for her/his valuable comments and suggestions
which lead to the final version of this paper.

\bibliographystyle{amsplain}

%%%%%%%%%%%%%%%%%%%%%%%%%%%%%%%%%%%%%%%%%%%%%%%%%%%%%%%%%%%%%%%%%%
%%%%%%%%%%%%%%%%%%%%%%%%%%%%%%%%%%%%%%%%%%%%%%%%%%%%%%%%%%%%%%%%%%

\end{document}